\documentclass[12pt,a4paper]{article}
\usepackage{amsfonts}
\usepackage{amssymb}
\usepackage{mathrsfs}
\usepackage{amsmath}
\usepackage{amsthm}
\usepackage{enumerate}

\allowdisplaybreaks
\newtheorem{defn}{Definition}[section]
\newtheorem{thm}[defn]{Theorem}
\newtheorem{lem}[defn]{Lemma}
\newtheorem{prop}[defn]{Proposition}
\newtheorem{cor}[defn]{Corollary}

\def\ad{{\rm ad}}
\def\dim{{{\rm dim}}}

\def\ch{{{\rm ch}}}

\begin{document}
\title{{\bf The construction of 3-Bihom-Lie algebras}}
\author{\normalsize \bf Juan Li, Liangyun Chen}
\date{{{\small{ School of Mathematics and Statistics,
  Northeast Normal University,\\
   Changchun 130024, China}}}} \maketitle
\date{}

{{\bf\begin{center}{Abstract}\end{center}}

The purpose of this paper is to study the construction of
$3$-Bihom-Lie algebras. We give some ways of constructing $3$-Bihom-Lie algebras from $3$-Bihom-Lie algebras and $3$-totally Bihom-associative algebras. Furthermore, we introduce $T_\theta$-extensions and $T_\theta^*$-extensions of $3$-Bihom-Lie algebras and prove the necessary and sufficient
conditions for a $2n$-dimensional quadratic $3$-Bihom-Lie algebra to be isomorphic to
a $T_\theta^*$-extension.



\renewcommand{\thefootnote}{\fnsymbol{footnote}}
\footnote[0]{ Corresponding author (L. Chen): chenly640@nenu.edu.cn.}
\footnote[0]{ Supported by NNSF of China (No. 11771069) and NSF of Jilin province (No. 20170101048JC).}

%----------------------------------------------------------------------------------------------------------------------
 \section{Introduction}
%----------------------------------------------------------------------------------------------------------------------

Bihom-algebras first appeared in \cite{GMMP} while investigating categorical of Hom-algebras. This class of algebras is an algebra with two homomorphisms $\alpha$ and $\beta$. When $\alpha=\beta$, Bihom-algebras will be return to Hom-algebras. And when $\alpha=\beta$ are identity maps, Bihom-algebras will be return to algebras. Since then many authors are interested in Bihom-algebras, such as Bihom-Lie algebras, Bihom-Lie superalgebras, Bihom-Lie colour algebras  in \cite{KAA, CQ, JLY, WG}. In particular, the definition of $n$-Bihom-Lie algebras and $n$-Bihom-Associative algebras were introduced in \cite{AAS}. Then the generalized derivations of $3$-BiHom-Lie superalgebras and $n$-BiHom-Lie algebras are studied in \cite{AMIA, ASO}.

The notion of the $T_\theta^*$-extension of Lie
algebras was introduced by Bordemann in \cite{MB}. It is one of the main
tools to prove that every symplectic quadratic Lie algebra is a special symplectic Manin
algebra in \cite{BBM}. Then this idea is
applied to generalize other algebras and the resulting algebras are studied, for example, $3$-Lie algebras, $3$-Hom-Lie algebras, Hom-Lie algebras and Bihom-Lie superalgebras etc. in \cite{RWYZ, BM, JLB, LCM, YLY, ZCM}. In this paper, we study the $T_\theta^*$-extension of $3$-Bihom-Lie algebras.

The paper is organized as follows. In Section 2 we recall the definition of $3$-Bihom-Lie algebras, and show that a $3$-Bihom-Lie algebra is given by the direct sum of two $3$-Bihom-Lie algebras and the tensor product of a $3$-totally Bihom-associative algebra and a $3$-Bihom-Lie algebra. Also we prove that a  homomorphism between $3$-Bihom-Lie algebras is a morphism if and only if
its graph is a Bihom subalgebra. In Section 3 we give the definition of representations of $3$-Bihom-Lie algebras. We can obtain the semidirect product $3$-Bihom-Lie algebra $(L\oplus M, [\cdot, \cdot, \cdot]_{\rho}, \alpha+\alpha_M,\beta+\beta_M)$ associated with any representation $\rho$ of a $3$-Bihom-Lie algebra $(L, [\cdot, \cdot, \cdot], \alpha, \beta)$ on $M$. And we can get a $T_\theta$-extension of $(L, [\cdot, \cdot, \cdot], \alpha, \beta)$ by a $3$-cocycle $\theta$. In Section 4 $T_\theta^*$-extensions of $3$-Bihom-Lie algebras are studied. We give the necessary and sufficient
conditions for a $2n$-dimensional quadratic $3$-Bihom-Lie algebra to be isomorphic to a $T_\theta^*$-extension.

%----------------------------------------------------------------------------------------------------------------------
 \section{Definitions and derivations of $3$-Bihom-Lie algebras}
%----------------------------------------------------------------------------------------------------------------------
Inspired by \cite{AMS, YD}, we give some construction of $3$-Bihom-Lie algebras, and similar conclusions are deduced.
\begin{defn}\cite{AAS}
A $3$-Bihom-Lie algebra over a field $\mathbb{K}$ is a
$4$-tuple $(L,[\cdot,\cdot,\cdot],\alpha,\beta)$, where $L$ is a vector
space, $3$-linear operator $[\cdot,\cdot,\cdot]
\colon L\times L\times L\rightarrow L$ and two linear maps $\alpha,\beta \colon L\rightarrow L$ satisfying the following conditions, $\forall\, x,y,z,u,v\in L$,
\begin{enumerate}[(1)]
\item $\alpha\circ\beta=\beta\circ\alpha$,
\item $\alpha ([x,y,z])=[\alpha(x),\alpha(y),\alpha(z)]$ and
  $\beta([x,y,z])=[\beta(x) ,\beta(y),\beta(z)]$,
\item Bihom-skewsymmetry: $[\beta(x),\beta(y),\alpha(z)]=-[\beta(y),\beta(x),\alpha(z)]=-[\beta(x),\beta(z),\alpha(y)],$
\item $3$-BiHom-Jacobi identity:
\begin{eqnarray*}
 &&[\beta^{2}(u),\beta^{2}(v),[\beta(x),\beta(y),\alpha(z)]]\\
 &=&[\beta^{2}(y),\beta^{2}(z),[\beta(u),\beta(v),\alpha(x)]]-[\beta^{2}(x),\beta^{2}(z),[\beta(u),\beta(v),\alpha(y)]]\\
 & &+[\beta^{2}(x),\beta^{2}(y),[\beta(u),\beta(v),\alpha(z)]].
\end{eqnarray*}
\end{enumerate}
\end{defn}
A $3$-Bihom-Lie algebra is called a \textit{regular} $3$-Bihom-Lie algebra if $\alpha$ and $\beta$ are algebra automorphisms.

Obviously, a $3$-Hom-Lie algebra $(L,[\cdot,\cdot,\cdot],\alpha)$ is a particular case of $3$-Bihom-Lie algebras, namely, $(L,[\cdot,\cdot,\cdot],\alpha,\alpha)$. Conversely, a $3$-Bihom-Lie algebra $(L,[\cdot,\cdot,\cdot],\alpha,\alpha)$ with isomorphism $\alpha$ is a $3$-Hom-Lie algebra $(L,[\cdot,\cdot,\cdot],\alpha)$.

\begin{defn}\cite{AAS}
 A sub-vector space $\eta\subseteq L$ is a Bihom subalgebra of $(L,[\cdot,\cdot, \cdot],\alpha,\beta)$ if $\alpha(\eta)\in \eta$, $\beta(\eta)\in \eta$ and $[x,y,z]\in \eta,~\forall\, x,y,z\in \eta$.
  It is said to be a Bihom ideal of $(L,[\cdot,\cdot,\cdot],\alpha,\beta)$ if $\alpha(\eta)\in \eta$, $\beta(\eta)\in \eta$ and $[x,y,z]\in \eta,~\forall\, x\in \eta,y,z\in L.$
\end{defn}

\begin{prop}\cite[Theorem 1.12]{AAS}
Let $(L, [\cdot, \cdot, \cdot])$ be a $3$-Lie algebra, and $\alpha, \beta : L\rightarrow L$ be algebraic homomorphisms such that $\alpha\circ\beta=\beta\circ\alpha$. The $(L, [\cdot,\cdot, ]_{\alpha\beta}, \alpha, \beta)$, where $[\cdot,\cdot, ]_{\alpha\beta}$ is defined
by
$[x, y, z]_{\alpha\beta}=[\alpha(x), \alpha(y), \beta(z)]$, $\forall\, x,y,z\in L$,
is a $3$-BiHom-Lie algebra.
\end{prop}

\begin{prop}
\label{ppp}
Let $(L,[\cdot,\cdot,\cdot], \alpha, \beta)$ be a $3$-BiHom-Lie algebra,  $\alpha^{'} ,\beta^{'} \colon L\rightarrow L$ be two algebraic homomorphisms and any two of the maps $\alpha, \beta, \alpha^{'} ,\beta^{'}$ commute. Then $(L, [\cdot,\cdot,\cdot]_{\alpha^{'},\beta^{'}}$ $:=[\cdot,\cdot,\cdot]\circ(\alpha^{'}\otimes\alpha^{'}\otimes\beta^{'}),\alpha\circ\alpha^{'} ,
\beta\circ\beta^{'} )$ is a $3$-Bihom-Lie algebra.
\end{prop}

\begin{proof}
First we check that the bracket product $[\cdot,\cdot,\cdot]_{\alpha^{'},\beta^{'}}$ is compatible with the structure
maps $\alpha\circ\alpha^{'}$ and
$\beta\circ\beta^{'} $. For any  $x, y, z \in L$, we have
\begin{eqnarray*}
[\beta\circ\beta^{'}(x),\beta\circ\beta^{'}(y),\alpha\circ\alpha^{'}(z)]_{\alpha^{'},\beta^{'}}&=& [\alpha^{'}\circ\beta\circ\beta^{'}(x),\alpha^{'}\circ\beta\circ\beta^{'}(y),\beta^{'}\circ\alpha\circ\alpha^{'}(z)]\\
&=& \alpha^{'}\circ\beta^{'}([\beta(x),\beta(y),\alpha(z)])\\
&=& -\alpha^{'}\circ\beta^{'}([\beta(y),\beta(x),\alpha(z)])\\
&=& -[\alpha^{'}\circ\beta^{'}\circ\beta(y),\alpha^{'}\circ\beta^{'}\circ\beta(x),\alpha^{'}\circ\beta^{'}\circ\alpha(z)]\\
&=& -[\beta\circ\beta^{'}(y),\beta\circ\beta^{'}(x),\alpha\circ\alpha^{'}(z)]_{\alpha^{'},\beta^{'}}.
\end{eqnarray*}
In the same way, $[\beta\circ\beta^{'}(x),\beta\circ\beta^{'}(y),\alpha\circ\alpha^{'}(z)]_{\alpha^{'},\beta^{'}}=-[\beta\circ\beta^{'}(x),\beta\circ\beta^{'}(z),\alpha\circ\alpha^{'}(y)]_{\alpha^{'},\beta^{'}}$.

Now we prove the $3$-Bihom-Jacobi identity. For any $x, y, z, u, v\in L$, we have
\begin{eqnarray*}
 &&[(\beta\circ\beta^{'})^2(u),(\beta\circ\beta^{'})^2(v),[\beta\circ\beta^{'}(x),\beta\circ\beta^{'}(y),\alpha\circ\alpha^{'}(z)]_{\alpha^{'},\beta^{'}}]_{\alpha^{'},\beta^{'}}\\
&=& [(\beta\circ\beta^{'})^2(u),(\beta\circ\beta^{'})^2(v), \alpha^{'}\circ\beta^{'}([\beta(x),\beta(y),\alpha(z)])]_{\alpha^{'},\beta^{'}}\\
&=& [\alpha^{'}\circ(\beta\circ\beta^{'})^2(u),\alpha^{'}\circ(\beta\circ\beta^{'})^2(v), \alpha^{'}\circ(\beta^{'})^2([\beta(x),\beta(y),\alpha(z)])]\\
&=& \alpha^{'}\circ(\beta^{'})^2([\beta^2(u),\beta^2(v),[\beta(x),\beta(y),\alpha(z)]]).
\end{eqnarray*}
By the $3$-Bihom-Jacobi identity of $[\cdot,\cdot,\cdot]$, we can obtain $[\cdot,\cdot,\cdot]_{\alpha^{'},\beta^{'}}$ satisfies the $3$-Bihom-Jacobi identity.

Thus, $(L, [\cdot,\cdot,\cdot]_{\alpha^{'},\beta^{'}}:=[\cdot,\cdot,\cdot]\circ(\alpha^{'}\otimes\alpha^{'}\otimes\beta^{'}), \alpha\circ\alpha^{'},
\beta\circ\beta^{'} )$ is a $3$-Bihom-Lie algebra.
\end{proof}

\begin{cor}
Let $(L,[\cdot,\cdot,\cdot], \alpha, \beta)$ be a $3$-BiHom-Lie algebra. Then $(L, [\cdot,\cdot,\cdot]_k:=[\cdot,\cdot,\cdot]\circ(\alpha^{k}\otimes\alpha^{k}\otimes\beta^{k}), \alpha^{k+1} ,
\beta^{k+1})$ is a $3$-Bihom-Lie algebra.
\end{cor}
\begin{proof}
Apply Proposition \ref{ppp} with $\alpha^{'}=\alpha^k$ and $\beta^{'}=\beta^k$.
\end{proof}

\begin{defn}\cite{AAS}
A $3$-totally BiHom-associative algebra is a vector space $A$ together with
a $3$-linear map $\mu:A\times A\times A\rightarrow A$ and two linear maps $\alpha, \beta:A\rightarrow A$, with notation $\mu(a_1,a_2,a_3)=a_1a_2a_3$, satisfying the following conditions, $\forall\, a_1,a_2,a_3,a_4,a_5\in A$,
\begin{enumerate}[(1)]
\item $\alpha\circ\beta=\beta\circ\alpha$,
\item $\alpha (a_1a_2a_3)=\alpha(a_1)\alpha(a_2)\alpha(a_3)$ and
  $\beta (a_1a_2a_3)=\beta(a_1)\beta(a_2)\beta(a_3)$,
\item $(a_1a_2a_3)\beta(a_4)\beta(a_5)=\alpha(a_1)(a_2a_3a_4)\beta(a_5)=\alpha(a_1)\alpha(a_2)(a_3a_4a_5).$
\end{enumerate}
\end{defn}

\begin{prop}
Let $(A, \mu, \alpha_1, \beta_1)$ be a $3$-totally BiHom-associative algebra and $(L,[\cdot,\cdot,\cdot],\\ \alpha_2, \beta_2)$ a $3$-BiHom-Lie algebra. If $\alpha_1$ is surjective and $\forall\, a_1,a_2,a_3\in A$,
\begin{eqnarray}\label{asso}
\beta_1(a_1)\beta_1(a_2)\alpha_1(a_3)=\beta_1(a_2)\beta_1(a_1)\alpha_1(a_3)=\beta_1(a_1)\beta_1(a_3)\alpha_1(a_2).
\end{eqnarray}
Then $(A\otimes L, [\cdot ,\cdot, \cdot]_{A\otimes L}, \alpha, \beta)$ is a $3$-BiHom-Lie algebra, where the $3$-linear map
 $[\cdot, \cdot, \cdot]_{A\otimes L} :
 \wedge^{3}(A\otimes L)\rightarrow A\otimes L$  is given by
$$[a_1\otimes x_1, a_2\otimes x_2, a_3\otimes x_3]_{A\otimes L}
= a_1a_2a_3\otimes[x_{1}, x_{2}, x_{3}],
~\forall\, a_{i}\in A, x_{i}\in L,i=1,2,3,$$
and the two linear maps $\alpha,
\beta:A\otimes L
\rightarrow A\otimes L$ are given by
$\alpha(a_1\otimes x_1)=\alpha_1(a_1)\otimes\alpha_2(x_1)$ and
$\beta(a_1\otimes x_1)=\beta_1(a_1)\otimes\beta_2(x_1).$
\end{prop}
\begin{proof}
Since $\alpha_1\beta_1=\beta_1\alpha_1$ and $\alpha_2\beta_2=\beta_2\alpha_2$, we have $\alpha\beta=\beta\alpha$. Meanwhile because $\alpha_1,\beta_1,\alpha_2,\beta_2$ are algebraic homomorphisms, that shows $\alpha$ and $\beta$ are algebraic homomorphisms.

Next we prove $[\cdot,\cdot, \cdot]_{A\otimes L}$ satisfies Bihom-skewsymmetry, \begin{eqnarray*}
&&[\beta(a_1\otimes x_1),\beta(a_2\otimes x_2), \alpha(a_3\otimes x_3)]_{A\otimes L}\\
&=&[\beta_1(a_1)\otimes\beta_2(x_1), \beta_1(a_2)\otimes\beta_2(x_2), \alpha_1(a_3)\otimes\alpha_2(x_3)]_{A\otimes L}\\
&=&\beta_1(a_1)\beta_1(a_2)\alpha_1(a_3)\otimes[\beta_2(x_1), \beta_2(x_2), \alpha_2(x_3)]\\
&=&-\beta_1(a_2)\beta_1(a_1)\alpha_1(a_3)\otimes[\beta_2(x_2), \beta_2(x_1), \alpha_2(x_3)]\\
&=&-[\beta(a_2\otimes x_2),\beta(a_1\otimes x_1), \alpha(a_3\otimes x_3)]_{A\otimes L}.
\end{eqnarray*}
Similarly, we can get $$[\beta(a_1\otimes x_1),\beta(a_2\otimes x_2), \alpha(a_3\otimes x_3)]_{A\otimes L}=-[\beta(a_1\otimes x_1),\beta(a_3\otimes x_3), \alpha(a_2\otimes x_2)]_{A\otimes L}.$$

Finally we prove the $3$-Bihom-Jacobi identity,
\begin{eqnarray*}
&&[\beta^2(a_4\otimes x_4), \beta^2(a_5\otimes x_5), [\beta(a_1\otimes x_1), \beta(a_2\otimes x_2), \alpha(a_3\otimes x_3)]_{A\otimes L}]_{A\otimes L}\\
&&-[\beta^2(a_3\otimes x_3), \beta^2(a_5\otimes x_5), [\beta(a_1\otimes x_1), \beta(a_2\otimes x_2), \alpha(a_4\otimes x_4)]_{A\otimes L}]_{A\otimes L}\\
&&+[\beta^2(a_3\otimes x_3), \beta^2(a_4\otimes x_4), [\beta(a_1\otimes x_1), \beta(a_2\otimes x_2), \alpha(a_5\otimes x_5)]_{A\otimes L}]_{A\otimes L}\\
&=&\beta_{1}^{2}(a_4)\beta_{1}^{2}(a_5)\big(\beta_{1}(a_1)\beta_{1}(a_2)\alpha_{1}(a_3)\big)\otimes[\beta_{2}^{2}(x_4), \beta_{2}^{2}(x_5), [\beta_{2}(x_1), \beta_{2}(x_2), \alpha_{2}(x_3)]]\\
&&-\beta_{1}^{2}(a_3)\beta_{1}^{2}(a_5)\big(\beta_{1}(a_1)\beta_{1}(a_2)\alpha_{1}(a_4)\big)\otimes[\beta_{2}^{2}(x_3), \beta_{2}^{2}(x_5), [\beta_{2}(x_1), \beta_{2}(x_2), \alpha_{2}(x_4)]]\\
&&+\beta_{1}^{2}(a_3)\beta_{1}^{2}(a_4)\big(\beta_{1}(a_1)\beta_{1}(a_2)\alpha_{1}(a_5)\big)\otimes[\beta_{2}^{2}(x_3), \beta_{2}^{2}(x_4), [\beta_{2}(x_1), \beta_{2}(x_2), \alpha_{2}(x_5)]]\\
&=&\beta_{1}^{2}(a_1)\beta_{1}^{2}(a_2)\big(\beta_{1}(a_3)\beta_{1}(a_4)\alpha_{1}(a_5)\big)\otimes\big([\beta_{2}^{2}(x_4), \beta_{2}^{2}(x_5), [\beta_{2}(x_1), \beta_{2}(x_2), \alpha_{2}(x_3)]]\\
&&-[\beta_{2}^{2}(x_3), \beta_{2}^{2}(x_5), [\beta_{2}(x_1), \beta_{2}(x_2), \alpha_{2}(x_4)]]\!+\![\beta_{2}^{2}(x_3), \beta_{2}^{2}(x_4), [\beta_{2}(x_1), \beta_{2}(x_2), \alpha_{2}(x_5)]]\big)\\
&=&\beta_{1}^{2}(a_1)\beta_{1}^{2}(a_2)\big(\beta_{1}(a_3)\beta_{1}(a_4)\alpha_{1}(a_5)\big)\otimes[\beta_{2}^{2}(x_1), \beta_{2}^{2}(x_2), [\beta_{2}(x_3), \beta_{2}(x_4), \alpha_{2}(x_5)]]\\
&=&[\beta^2(a_1\otimes x_1), \beta^2(a_2\otimes x_2), [\beta(a_3\otimes x_3), \beta(a_4\otimes x_4), \alpha(a_5\otimes x_5)]_{A\otimes L}]_{A\otimes L},
\end{eqnarray*}
where using \eqref{asso} in the second equality.

Thus $(A\otimes L, [\cdot ,\cdot, \cdot]_{A\otimes L}, \alpha, \beta)$ is a $3$-BiHom-Lie algebra.
\end{proof}

In \cite{AAS}, $3$-BiHom-Lie algebras can be induced by BiHom-Lie algebras. Now we can get a Bihom-Lie algebra by a $3$-Bihom-Lie algebra.
\begin{prop}
Let $(L,[\cdot,\cdot,\cdot], \alpha, \beta)$ be a $3$-BiHom-Lie algebra. Suppose $a\in L$ satisfies $\alpha(a)=\beta(a)=a$. Then $(L,[\cdot,\cdot], \alpha, \beta)$ is a Bihom-Lie algebra, where $[x,y]=[a,x,y], ~\forall\, x,y\in L$.
\end{prop}
\begin{proof}
First we prove $[\cdot,\cdot]$ satisfies Bihom-skewsymmetry, $$[\beta(x), \alpha(y)]=[a, \beta(x), \alpha(y)]=[\beta(a), \beta(x), \alpha(y)]=-[\beta(a), \beta(y), \alpha(x)]=-[\beta(y), \alpha(x)].$$
Next we prove the Bihom-Jacobi identity,
\begin{eqnarray*}
&&[\beta^2(x), [\beta(y), \alpha(z)]]+[\beta^2(y), [\beta(z), \alpha(x)]]+[\beta^2(z), [\beta(x), \alpha(y)]]\\
&=&[a, \beta^2(x), [a, \beta(y), \alpha(z)]]+[a, \beta^2(y), [a, \beta(z), \alpha(x)]]+[a, \beta^2(z), [a, \beta(x), \alpha(y)]]\\
&=&[\beta^2(a), \beta^2(x), [\beta(a), \beta(y), \alpha(z)]]+[\beta^2(a), \beta^2(y), [\beta(a), \beta(z), \alpha(x)]]\\
& &+[\beta^2(a), \beta^2(z), [\beta(a), \beta(x), \alpha(y)]]\\
&=&[\beta^2(y), \beta^2(z), [\beta(a), \beta(x), \alpha(a)]]-[\beta^2(a), \beta^2(z), [\beta(a), \beta(x), \alpha(y)]]\\
& &+[\beta^2(a), \beta^2(y), [\beta(a), \beta(x), \alpha(z)]]+[\beta^2(a), \beta^2(y), [\beta(a), \beta(z), \alpha(x)]]\\
& &+[\beta^2(a), \beta^2(z), [\beta(a), \beta(x), \alpha(y)]]\\
&=&0.
\end{eqnarray*}
Thus, $(L,[\cdot,\cdot], \alpha, \beta)$ is a Bihom-Lie algebra.
\end{proof}

\begin{prop}
Given two $3$-Bihom-Lie algebras $( L, [\cdot,\cdot,\cdot], \alpha,
\beta )$ and $(L^{\prime}, [\cdot,\cdot,\cdot]^{\prime},\alpha^{\prime},
\beta^{\prime})$, there is a $3$-Bihom-Lie
algebra $(L\oplus L^{\prime}, [\cdot, \cdot, \cdot]_{L\oplus L^{\prime}},
\alpha+\alpha^{\prime}, \beta+\beta^{\prime} ),$
where the $3$-linear map
 $[\cdot, \cdot, \cdot]_{L\oplus L^{\prime}} :
 \wedge^{3}(L\oplus L^{\prime})\rightarrow L\oplus L^{\prime}$  is given by
$$[u_{1}+v_{1}, u_{2}+v_{2}, u_{3}+v_{3}]_{L\oplus L^{\prime}}
= [u_{1}, u_{2}, u_{3}]+[v_{1}, v_{2}, v_{3}]^{\prime},
~\forall\, u_{i}\in L, v_{i}\in L^{\prime},i=1,2,3,$$
and the two linear maps $\alpha+\alpha^{\prime},
\beta+\beta^{\prime}:L\oplus L^{\prime}
\rightarrow L\oplus L^{\prime}$ are given by
$$(\alpha+\alpha^{\prime})(u+v)=\alpha(u)+\alpha^{\prime}(v),$$
$$(\beta+\beta^{\prime})(u+v)=\beta(u)+\beta^{\prime}(v), \forall\ u\in L, v\in L^{\prime}.$$
\end{prop}

\begin{proof}
Since $\alpha,\beta,\alpha^{\prime},\beta^{\prime}$ are algebraic homomorphisms, that shows $\alpha+\alpha^{\prime}$ and
$\beta+\beta^{\prime}$ are algebraic homomorphisms.
For any $u_i\in L,~v_i\in L^{\prime}$, $i=1,2,3,4,5$, we have
\begin{eqnarray*}
(\alpha+\alpha^{\prime})\circ(\beta+\beta^{\prime})(u_1+v_1)&=&(\alpha+\alpha^{\prime})(\beta(u_1)+\beta^{\prime}(v_1))\\
&=&\alpha\circ \beta(u_1)+\alpha^{\prime}\circ\beta^{\prime}(v_1)\\
&=&\beta\circ\alpha(u_1)+\beta^{\prime}\circ\alpha^{\prime}(v_1)\\
&=&(\beta+\beta^{\prime})\circ(\alpha+\alpha^{\prime})(u_1+v_1),
\end{eqnarray*}
i.e. $(\alpha+\alpha^{\prime})\circ(\beta+\beta^{\prime})=(\beta+\beta^{\prime})\circ(\alpha+\alpha^{\prime})$.

Next, we consider the Bihom-skewsymmetry,
\begin{eqnarray*}
&&{[(\beta+\beta^{\prime})(u_1+v_1),(\beta+\beta^{\prime})(u_2+v_2), (\alpha+\alpha^{\prime})(u_3+v_3)]}_{L\oplus L'}\\
&=&{[\beta(u_1)+\beta^{\prime}(v_1), \beta(u_2)+\beta^{\prime}(v_2), \alpha(u_3)+\alpha^{\prime}(v_3)]}_{L\oplus L'}\\
&=&[\beta(u_{1}), \beta(u_{2}), \alpha(u_{3})]+ [\beta^{\prime}(v_{1}),\beta^{\prime}(v_{2}), \alpha^{\prime}(v_{3})]^{\prime}\\
&=&-[\beta(u_{2}),\beta(u_{1}), \alpha(u_{3})]-[\beta^{\prime}(v_{2}),\beta^{\prime}(v_{1}), \alpha^{\prime}(v_{1})]^{\prime}\\
&=&-{[(\beta+\beta^{\prime})(u_2+v_2), (\beta+\beta^{\prime})(u_1+v_1), (\alpha+\alpha^{\prime})(u_3+v_3)]}_{L\oplus L'}.
\end{eqnarray*}
Similarly, we can get
\begin{eqnarray*}&&{[(\beta+\beta^{\prime})(u_1+v_1),(\beta+\beta^{\prime})(u_2+v_2), (\alpha+\alpha^{\prime})(u_3+v_3)]}_{L\oplus L'}\\
&=&-{[(\beta+\beta^{\prime})(u_1+v_1),(\beta+\beta^{\prime})(u_3+v_3), (\alpha+\alpha^{\prime})(u_2+v_2)]}_{L\oplus L'}.
\end{eqnarray*}
Finally, we prove the $3$-Bihom-Jacobi identity,
\begin{eqnarray*}
&&[(\beta+\beta^{\prime})^2(u_1+v_1), (\beta+\beta^{\prime})^2(u_2+v_2),[(\beta+\beta^{\prime})(u_3+v_3), (\beta+\beta^{\prime})(u_4+v_4),\\
 &&~(\alpha+\alpha^{\prime})(u_5+v_5)]_{L\oplus L'}]_{L\oplus L'}\\
&=&[\beta^2(u_1)\!+\!{\beta^{\prime}}^2(v_1), \beta^2(u_2)\!+\!{\beta^{\prime}}^2(v_2),[\beta(u_3),\beta(u_4), \alpha(u_5)]\!+\![\beta^{\prime}(v_3), \beta^{\prime}(v_4),\alpha^{\prime}(v_5)]^{\prime}]_{L\oplus L'}\\
&=&[\beta^2(u_1), \beta^2(u_2),[\beta(u_3), \beta(u_4), \alpha(u_4)]]+[{\beta^{\prime}}^2(v_1),{\beta^{\prime}}^2(v_2), [\beta^{\prime}(v_3), \beta^{\prime}(v_4),\alpha^{\prime}(v_5)]^{\prime}]^{\prime}.
\end{eqnarray*}
By the $3$-Bihom-Jacobi identity of $[\cdot,\cdot,\cdot]$ and $[\cdot,\cdot,\cdot]^{\prime}$, $[\cdot,\cdot,\cdot]_{L\oplus L'}$ satisfies the $3$-Bihom-Jacobi identity.

Thus, $(L\oplus L^{\prime}, [\cdot, \cdot, \cdot]_{L\oplus L^{\prime}},
\alpha+\alpha^{\prime}, \beta+\beta^{\prime} )$ is a $3$-Bihom-Lie
algebra.
\end{proof}

\begin{defn}\cite{AAS}
Let $( L, [\cdot,\cdot,\cdot] ,\alpha,
\beta )$ and $(L^{\prime}, [\cdot,\cdot,\cdot]^{\prime} ,\alpha^{\prime},
\beta^{\prime})$ be two $3$-Bihom-Lie algebras. A homomorphism $f:L \rightarrow L'$ is said to be a morphism of $3$-Bihom-Lie algebras if
\begin{equation*} f([x, y, z])=[f(x), f(y), f(z)]^{\prime}, ~\forall\, x,y,z\in L, \end{equation*}
\begin{equation*}f\circ\alpha=\alpha^{\prime}\circ f, \qquad\qquad\qquad\qquad\end{equation*}
\begin{equation*}f\circ\beta=\beta^{\prime}\circ f. \qquad\qquad\qquad\qquad\end{equation*}

Denote by $\phi_{f}=\{x+f(x)~|~x\in L\}\subset L\oplus L'$ which is the graph of a linear map $f:L \rightarrow L'$.
\end{defn}

\begin{prop}
 A homomorphism $f: ( L, [\cdot,\cdot,\cdot] ,\alpha ,\beta )\rightarrow(L^{\prime},
 [\cdot,\cdot,\cdot]^{\prime} ,\alpha^{\prime} ,\beta^{\prime})$ is a morphism of
  $3$-Bihom-Lie algebras if and only if the graph
  $\phi_{f}\subset L\oplus L^{\prime}$ is a Bihom-subalgebra
  of $(L\oplus L^{\prime}, [\cdot,\cdot, \cdot]_{L\oplus L^{\prime}},
  \alpha+\alpha^{\prime}, \beta+\beta^{\prime} ).$
\end{prop}

\begin{proof}
Let $f: ( L, [\cdot,\cdot,\cdot] ,\alpha ,\beta )\rightarrow(L^{\prime}, [\cdot,\cdot,\cdot]^{\prime},
\alpha^{\prime} ,\beta^{\prime})$ is a morphism of $3$-Bihom-Lie algebras,
 for any $u, v, w\in L,$  we have
$$[u+f(u), v+f(v)
,w+f(w)]_{L\oplus L^{\prime}}=[u,v,w]+[f(u),
f(v),f(w)]^{\prime}=[u,v,w]+f([u,v,w]).$$
Thus the graph $\phi_{f}$ is closed under the bracket
operation $[\cdot,\cdot,\cdot]_{L\oplus L^{\prime}}$.  Furthermore, we have
$(\alpha+\alpha^{\prime})(u+f(v))=\alpha(u)+\alpha^{\prime}\circ f(v)=\alpha(u)+ f\circ\alpha(v)$,
which implies that $(\alpha+\alpha^{\prime})(\phi_{f})\subset\phi_{f}$.
Similarly, $(\beta+\beta^{\prime})(\phi_{f})\subset\phi_{f}$.
Thus $\phi_{f}$ is a Bihom-subalgebra of $(L\oplus L^{\prime},
[\cdot, \cdot]_{L\oplus L^{\prime}}, \alpha+\alpha^{\prime}, \beta+\beta^{\prime} ).$

Conversely, if the graph $\phi_{f}$ is a Bihom-subalgebra of
  $(L\oplus L^{\prime}, [\cdot, \cdot]_{L\oplus L^{\prime}},
  \alpha+\alpha^{\prime}, \beta+\beta^{\prime} )$, we have $[u+f(u), v+
  f(v), w+f(w)]_{L\oplus L^{\prime}}=[u, v, w]+
  [f(u), f(v), f(w)]^{\prime}\in\phi_{f}$,
 which implies that $[f(u), f(v), f(v)]^{\prime}=f([u, v, w])$.
 Furthermore, $(\alpha+\alpha^{\prime})(\phi_{f})\subset\phi_{f}$ yields that
  $(\alpha+\alpha^{\prime})(u+f(u))=\alpha(u)+\alpha^{\prime}\circ f(u)\in\phi_{f}$,
  which is equivalent to the condition $\alpha^{\prime}\circ f(u)=f\circ\alpha(u)$, i.e. $\alpha^{\prime}\circ f=f\circ\alpha.$
  Similarly, $\beta^{\prime}\circ f=f\circ\beta.$
  Therefore, $f$ is a morphism of $3$-Bihom-Lie algebras.
\end{proof}

Let $( L,[\cdot,\cdot,\cdot] ,\alpha ,\beta )$ be a $3$-Bihom-Lie algebra. For any integers $k$ and $l$, denote by $\alpha^{k}$ the $k$-times composition of $\alpha$ and $\beta^{l}$ the $l$-times composition of $\beta$, i.e.
 $$\alpha^{k}=\alpha\circ\cdots\circ\alpha (k-times),~ \beta^{l}=\beta\circ\cdots\circ\beta(l-times).$$
 Since the maps $\alpha, \beta$ commute, we denote by $$\alpha^{k}\beta^{l}=\underbrace{\alpha\circ\cdots\circ\alpha}_{k-times}\circ\underbrace{\beta\circ\cdots\circ\beta}_{l-times}.$$
In particular, $\alpha^{0}\beta^{0}=Id, \alpha^{1}\beta^{1}=\alpha\beta$. If $( L,[\cdot,\cdot,\cdot] ,\alpha ,\beta ) $ be a regular $3$-Bihom-Lie algebra, $\alpha^{-k}\beta^{-l}$ is the inverse of $\alpha^{k}\beta^{l}.$

\begin{defn}\cite{AMIA}
For any integers $k$ and $l$,  a linear map $D: L\rightarrow L$ is called an $\alpha^{k}\beta^{l}$-derivation of the $3$-Bihom-Lie algebra $( L,[\cdot,\cdot,\cdot] ,\alpha ,\beta )$, if for all $u,v,w\in L$,
$$
D\circ\alpha=\alpha\circ D, \, D\circ\beta=\beta\circ D, $$
$$
D[u,v,w]\!=\![D(u), \alpha^{k}\beta^{l}(v), \alpha^{k}\beta^{l}(w)]\!+\![\alpha^{k}\beta^{l}(u), D(v), \alpha^{k}\beta^{l}(w)]\!+\![\alpha^{k}\beta^{l}(u), \alpha^{k}\beta^{l}(v), D(w)].
$$

\end{defn}
Denote by $Der_{\alpha^{k}\beta^{l}}(L)$ the set of
$\alpha^{k}\beta^{l}$-derivations of $( L,[\cdot,\cdot,\cdot] ,\alpha ,\beta ) $.

Now let $( L,[\cdot,\cdot,\cdot] ,\alpha ,\beta ) $ be a regular $3$-Bihom-Lie algebra, for any $u\in L$ satisfying
$\alpha(u_1)=\beta(u_1)=u_1$, $\alpha(u_2)=\beta(u_2)=u_1$, define $D_{k,l}(u_1, u_2)\in End(L)$ by
$$D_{k,l}(u_1,u_2)(w)=[u_1, u_2, \alpha^{k}\beta^{l}(w)], ~ \forall\, w\in L.$$

\begin{prop}
$D_{k,l}(u_1,u_2)$ is an $\alpha^{k}\beta^{l+1}$-derivation. We call an {inner} $\alpha^{k}\beta^{l+1}$-derivation.
\end{prop}
\begin{proof}
First we have
$$D_{k,l}(u_1, u_2)(\alpha(w))=[u_1, u_2, \alpha^{k}\beta^{l}\alpha(w) ]=\alpha([u_1, u_2, \alpha^{k}\beta^{l}(w) ])=\alpha\circ D_{k,l}(u_1,u_2)(w). $$
Similarly, $D_{k,l}(u_1, u_2)(\beta(w))=\beta\circ D_{k,l}(u_1,u_2)(w). $
Then, we have
\begin{eqnarray*}
&&D_{k,l}(u_1,u_1)([u,v,w])\\
&=&[u_1, u_2, \alpha^{k}\beta^{l}[u,v,w]]\\
&=&[\beta^{2}(u_1), \beta^{2}(u_2), [\beta\alpha^{k}\beta^{l-1}(u),\beta\alpha^{k}\beta^{l-1}(v),
\alpha\alpha^{k-1}\beta^{l}(w)]]\\
&=&[\beta^{2}\alpha^{k}\beta^{l-1}(v),\beta^{2}\alpha^{k-1}\beta^{l}(w),[\beta(u_1), \beta(u_2), \alpha\alpha^{k}\beta^{l-1}(u)]]\\
& &-[\beta^{2}\alpha^{k}\beta^{l-1}(u),\beta^{2}\alpha^{k-1}\beta^{l}(w),[\beta(u_1), \beta(u_2), \alpha\alpha^{k}\beta^{l-1}(v)]]\\
& &+[\beta^{2}\alpha^{k}\beta^{l-1}(u),\beta^{2}\alpha^{k}\beta^{l-1}(v),[\beta(u_1), \beta(u_2), \alpha\alpha^{k-1}\beta^{l}(w)]]\\
&=&[\beta\alpha^{k}\beta^{l}(v),\beta\alpha^{k-1}\beta^{l+1}(w),\alpha([u_1, u_2, \alpha^{k}\beta^{l-1}(u)])]\\
& &-[\beta\alpha^{k}\beta^{l}(u),\beta\alpha^{k-1}\beta^{l+1}(w),\alpha([u_1, u_2, \alpha^{k}\beta^{l-1}(v)])]\\
& &+[\beta\alpha^{k}\beta^{l}(u),\beta\alpha^{k}\beta^{l}(v),\alpha([u_1, u_2, \alpha^{k-1}\beta^{l}(w)])]\\
&=&[\beta([u_1, u_2, \alpha^{k}\beta^{l-1}(u)]), \beta\alpha^{k}\beta^{l}(v),\alpha\alpha^{k-1}\beta^{l+1}(w)]\\
& &+[\beta\alpha^{k}\beta^{l}(u),\beta([u_1, u_2, \alpha^{k}\beta^{l-1}(v)]), \alpha\alpha^{k-1}\beta^{l+1}(w)]\\
& &+[\beta\alpha^{k}\beta^{l}(u),\beta\alpha^{k}\beta^{l}(v),\alpha([u_1, u_2, \alpha^{k-1}\beta^{l}(w)])]\\
&=&[D_{k,l}(u_1,u_2)(u), \alpha^{k}\beta^{l+1}(v), \alpha^{k}\beta^{l+1}(w)]+[\alpha^{k}\beta^{l+1}(u),D_{k,l}(u_1,u_2)(v), \alpha^{k}\beta^{l+1}(w)]\\
& &+[\alpha^{k}\beta^{l+1}(u),\alpha^{k}\beta^{l+1}(v),D_{k,l}(u_1,u_2)(w)].
\end{eqnarray*}
Therefore, $D_{k,l}(u_1,u_2)$ is an $\alpha^{k}\beta^{l+1}$-derivation.
\end{proof}
Denote  by $Inn_{\alpha^{k}\beta^{l}}(L)$ the set of
inner $\alpha^{k}\beta^{l}$-derivations, i.e.
$$Inn_{\alpha^{k}\beta^{l}}(L)=\{[u_1, u_2, \alpha^{k}
\beta^{l-1}(\cdot)]\mid u_1,u_2\in L, \alpha(u_1)=\beta(u_1)=u_1, \alpha(u_2)= \beta(u_2)=u_2\}.$$

In \cite{AMIA}, we can find $Der(L)=\bigoplus\limits_{k,l}
Der_{\alpha^{k}\beta^{l}}(L)$ with $[D, D^{'}]=D \circ D^{'}-D^{'} \circ D$ is a Lie algebra. Set $Inn(L)=\bigoplus\limits_{k,l}
Inn_{\alpha^{k}\beta^{l}}(L)$.

\begin{prop}
$Inn(L)$ is an ideal of $Der(L)$.
\end{prop}
\begin{proof}
Let $D_{s,t}(u_1,u_2)\in Inn_{\alpha^{s}\beta^{t+1}}(L)$ and $D\in Der_{\alpha^{k}\beta^{l}}(L)$. Then $[D, D_{s,t}(u_1,u_2)]\in Der_{\alpha^{k+s}\beta^{l+t+1}}(L)$, and for any $y\in L$
\begin{eqnarray*}
&&[D, D_{s,t}(u_1,u_2)](y)\\
&=&D[u_1,u_2,\alpha^s\beta^t(y)]-[u_1,u_2,\alpha^s\beta^tD(y)]\\
&= &[D(u_1),\alpha^k\beta^l(u_2),\alpha^{k+s}\beta^{l+t}(y)]+[\alpha^k\beta^l(u_1),D(u_2),\alpha^{k+s}\beta^{l+t}(y)]\\
& &+[\alpha^k\beta^l(u_1),\alpha^k\beta^l(u_2),\alpha^{s}\beta^{t}D(y)]-[u_1,u_2,\alpha^s\beta^tD(y)]\\
&=&[D(u_1),u_2,\alpha^{k+s}\beta^{l+t}(y)]+[u_1,D(u_2),\alpha^{k+s}\beta^{l+t}(y)]\\
&=&D_{k+s,l+t}(D(u_1),u_2)(y)+D_{k+s,l+t}(u_1,D(u_2))(y).
\end{eqnarray*}
Therefore $[D, D_{s,t}(u_1,u_2)]\in Inn_{\alpha^{k+s}\beta^{l+t+1}}(L)$, i.e. $Inn(L)$ is an ideal of $Der(L)$.
\end{proof}

\section{Representations and $T_\theta$-extensions of $3$-Bihom-Lie algebras}

\begin{defn}
Let $(L, [\cdot,\cdot,\cdot], \alpha, \beta)$ be a $3$-Bihom-Lie algebra.
A representation of $L$ is a $4$-tuple
$(M,\rho, \alpha_{M}, \beta_{M})$, where $M$ is a vector space,
$\alpha_{M},\beta_{M}\in End(M)$ are two commuting linear maps
and $\rho: L\times L\rightarrow End(M)$ is a skewsymmetry bilinear map, such that for all $u,v, x, y\in L$,
\begin{enumerate}[(1)]
\item $\rho(\alpha(u),\alpha(v))\circ \alpha_M=\alpha_M\circ\rho(u,v),$
\item $\rho(\beta(u),\beta(u))\circ \beta_M=\beta_M\circ\rho(u,v),$
\item $\quad\rho(\alpha\beta(u),\alpha\beta(v))\circ\rho(x,y)\\
\!=\!\rho(\beta(x),\beta(y))\!\circ\!\rho(\alpha(u),\alpha(v))\!+\!\rho([\beta(u),\beta(v),x],\beta(y))\!\circ\!\beta_M\!+\!\rho(\beta(x),[\beta(u),\beta(v),y])\!\circ\!\beta_M,$
\item $\quad\rho([\beta(u), \beta(u), x], \beta(y))\circ\beta_M\\
=\rho(\alpha\beta(v),\beta(x))\circ\rho(\alpha(u),y)+\rho(\beta(x),\alpha\beta(u))\circ\rho(\alpha(v),y)+\rho(\alpha\beta(u),\alpha\beta(v))\circ\rho(x,y).$\label{repre3}
\end{enumerate}
\end{defn}

\begin{prop}
Let $(L, [\cdot,\cdot,\cdot], \alpha, \beta)$ be a regular $3$-Bihom-Lie algebra. $\ad: L\times L\rightarrow End(L)$ is a linear map such that $$\ad(u_1,u_2)(x)=[u_1, u_2, x], ~\forall\, u_1, u_2, x\in L.$$ Then $(L,\ad, \alpha, \beta)$ is a representation of $(L, [\cdot,\cdot,\cdot], \alpha, \beta)$, called adjoint representation.
\end{prop}
\begin{proof}
Follows a direct computation by the definition of representations.
\end{proof}

\begin{prop}\label{product}
Let $(L, [\cdot, \cdot, \cdot], \alpha,\beta)$ be a $3$-Bihom-Lie algebra and  $(M, \rho, \alpha_M, \beta_M)$ a
representation of $L$. Assume that the maps $\alpha$ and $\beta_M$ are surjective. Then $L\ltimes M:=(L\oplus M, [\cdot,\cdot,\cdot]_\rho,\alpha+ \alpha_M,\beta+\beta_M)$ is a $3$-Bihom-Lie algebra, where
$\alpha+\alpha_M,\beta+\beta_M:L\oplus M\rightarrow L\oplus M$ are defined by $(\alpha+\alpha_M)(u+x)=\alpha(u)+\alpha_M(x)$ and $(\beta+\beta_M)(u+x)=\beta(u)+\beta_M(x)$, and the bracket $[\cdot,\cdot,\cdot]_\rho$ is
defined by
\begin{eqnarray*}
&&[u+x, v+y, w+z]_{\rho}\\
&=&[u, v,z]+\rho(u,v)(z)-\rho(u, \alpha^{-1}\beta(w))(\alpha_M\beta_M^{-1}(y))+\rho(v, \alpha^{-1}\beta(w))(\alpha_M\beta_M^{-1}(x)),
\end{eqnarray*}
for all $u, v, w \in L$ and $x,y,z\in M$. We call $L\ltimes M$ the semidirect product of the $3$-Bihom-Lie algebra $(L, [\cdot,\cdot, \cdot], \alpha,\beta)$ and $M$.
\end{prop}
\begin{proof}
First we show $(\alpha+\alpha_M)\circ(\beta+\beta_M)=(\beta+\beta_M)\circ(\alpha+\alpha_M)$ from the fact $\alpha\circ\beta=\beta\circ\alpha, \alpha_M\circ\beta_M=\beta_M\circ\alpha_M$.

Then, we can obtain $(\alpha+\alpha_M)$ is a algebraic homomorphism,
\begin{eqnarray*}
&&[(\alpha+\alpha_M)(u+x),(\alpha+\alpha_M)(v+y),(\alpha+\alpha_M)(w+z)]_\rho\\
&=&[\alpha(u)+\alpha_M(x),\alpha(v)+\alpha_M(y),\alpha(w)+\alpha_M(z)]_\rho\\
&=&[\alpha(u),\alpha(v),\alpha(w)]+\rho(\alpha(u),\alpha(v))(\alpha_M(z))-\rho(\alpha(u),\alpha^{-1}\beta\alpha(w))(\alpha_M\beta^{-1}_M\alpha_M(y))\\
&&+\rho(\alpha(v),\alpha^{-1}\beta\alpha(w))(\alpha_M\beta^{-1}_M\alpha_M(x))\\
&=&\alpha([u,v,w])+\alpha_M\rho(u,v)(z)-\alpha_M\rho(u,\alpha^{-1}\beta(w))(\alpha_M\beta^{-1}_M(y))\\
&&+\alpha_M\rho(v,\alpha^{-1}\beta(w))(\alpha_M\beta^{-1}_M(x))\\
&=&(\alpha+\alpha_M)([u+x, v+y, w+z]_{\rho}).
\end{eqnarray*}
Similarly, $(\beta+\beta_M)$ is a algebraic homomorphism.

Next we show that $[\cdot,\cdot, \cdot]_{\rho}$ satisfies Bihom skewsymmetry,
\begin{eqnarray*}
&&[(\beta+\beta_M)(u+x), (\beta+\beta_M)(v+y), (\alpha+\alpha_M)(w+z)]_\rho\\
&=&[\beta(u)+\beta_M(x), \beta(v)+\beta_M(y), \alpha(w)+\alpha_M(z)]_\rho\\
&=&[\beta(u), \beta(v), \alpha(w)]+ \rho(\beta(u),\beta(v))(\alpha_M(z))-\rho(\beta(u), \alpha^{-1}\beta\alpha(w))(\alpha_M\beta_M^{-1}\beta_M(y))\\
&&+\rho(\beta(v), \alpha^{-1}\beta\alpha(w))(\alpha_M\beta_M^{-1}\beta_M(x))\\
&=&-[\beta(v), \beta(u), \alpha(w)]- \rho(\beta(v),\beta(u))(\alpha_M(z))+\rho(\beta(v), \alpha^{-1}\beta\alpha(w))(\alpha_M\beta_M^{-1}\beta_M(x))\\
&&-\rho(\beta(u), \alpha^{-1}\beta\alpha(w))(\alpha_M\beta_M^{-1}\beta_M(y))\\
&=&-[(\beta+\beta_M)(v+y), (\beta+\beta_M)(u+x), (\alpha+\alpha_M)(w+z)]_\rho.
\end{eqnarray*}
In the same way, we also have $[(\beta+\beta_M)(u+x), (\beta+\beta_M)(v+y), (\alpha+\alpha_M)(w+z)]_\rho=-[(\beta+\beta_M)(u+x), (\beta+\beta_M)(w+z), (\alpha+\alpha_M)(v+y)]_\rho$.

Finally, we can obtain for all $u_i\in L, x_i\in M, i=1,2,3,4,5,$
\begin{eqnarray*}
&&[(\beta+\beta_M)^2(u_1+x_1),(\beta+\beta_M)^2(u_2+x_2),[(\beta+\beta_M)(u_3+x_3),(\beta+\beta_M)(u_4+x_4),\\
&&\,(\alpha+\alpha_M)(u_5+x_5)]_\rho]_\rho\\
&=&[\beta^2(u_1)\!+\!\beta_M^2(x_1),\!\beta^2(u_2)\!+\!\beta_M^2(x_2),\![\beta(u_3)\!+\!\beta_M(x_3),\!\beta(u_4)\!+\!\beta_M(x_4),\!\alpha(u_5)\!+\!\alpha_M(x_5)]_\rho]_\rho\\
&=&[\beta^2(u_1)+\beta_M^2(x_1),\beta^2(u_2)+\beta_M^2(x_2),[\beta(u_3),\beta(u_4),\alpha(u_5)]+\rho(\beta(u_3),\beta(u_4))(\alpha_M(x_5))\\
& &-\rho(\beta(u_3),\beta(u_5))(\alpha_M(x_4))+\rho(\beta(u_4),\beta(u_5))(\alpha_M(x_3))]_\rho\\
&=&[\beta^2(u_1),\beta^2(u_2),[\beta(u_3),\beta(u_4),\alpha(u_5)]]+\rho(\beta^2(u_1),\beta^2(u_2))\big(\rho(\beta(u_3),\beta(u_4))(\alpha_M(x_5))\\
&&-\rho(\beta(u_3),\beta(u_5))(\alpha_M(x_4))+\rho(\beta(u_4),\beta(u_5))(\alpha_M(x_3))\big)\\
&&-\rho(\beta^2(u_1),\alpha^{-1}\beta([\beta(u_3),\beta(u_4),\alpha(u_5)]))(\alpha_M\beta_M(x_2))\\
&&+\rho(\beta^2(u_2),\alpha^{-1}\beta([\beta(u_3),\beta(u_4),\alpha(u_5)]))(\alpha_M\beta_M(x_1))\\
&=&[\beta^2(u_4),\beta^2(u_5),[\beta(u_1),\beta(u_2),\alpha(u_3)]]-[\beta^2(u_3),\beta^2(u_5),[\beta(u_1),\beta(u_2),\alpha(u_4)]]\\
&&+[\beta^2(u_3),\beta^2(u_4),[\beta(u_1),\beta(u_2),\alpha(u_5)]]+\rho(\beta^2(u_3),\beta^2(u_4))\rho(\beta(u_1),\beta(u_2))(\alpha_M(x_5))\\
&&+\rho(\alpha^{-1}\beta([\beta(u_1),\beta(u_2),\alpha(u_3)]),\beta^2(u_4))(\alpha_M\beta_M(x_5))\\
&&+\rho(\beta^2(u_3),\alpha^{-1}\beta([\beta(u_1),\beta(u_2),\alpha(u_4)]))(\alpha_M\beta_M(x_5))\\
&&-\rho(\beta^2(u_3),\beta^2(u_5))\rho(\beta(u_1),\beta(u_2))(\alpha_M(x_4))\\
&&-\rho(\alpha^{-1}\beta([\beta(u_1),\beta(u_2),\alpha(u_3)]),\beta^2(u_5))(\alpha_M\beta_M(x_4))\\
&&-\rho(\beta^2(u_3),\alpha^{-1}\beta([\beta(u_1),\beta(u_2),\alpha(u_5)]))(\alpha_M\beta_M(x_4))\\
&&+\rho(\beta^2(u_4),\beta^2(u_5))\rho(\beta(u_1),\beta(u_2))(\alpha_M(x_3))\\
&&+\rho(\alpha^{-1}\beta([\beta(u_1),\beta(u_2),\alpha(u_4)]),\beta^2(u_5))(\alpha_M\beta_M(x_3))\\
&&+\rho(\beta^2(u_4),\alpha^{-1}\beta([\beta(u_1),\beta(u_2),\alpha(u_5)]))(\alpha_M\beta_M(x_3))\\
&&+\rho(\beta^2(u_4)\!,\!\beta^2(u_5))\rho(\beta(u_3)\!,\!\beta(u_1))(\alpha_M(x_2))\!+\!\rho(\beta^2(u_5)\!,\!\beta^2(u_3))\rho(\beta(u_4)\!,\!\beta(u_1))(\alpha_M(x_2))\\
&&+\rho(\beta^2(u_3)\!,\!\beta^2(u_4))\rho(\beta(u_5)\!,\!\beta(u_1))(\alpha_M(x_2))\!-\!\rho(\beta^2(u_4)\!,\!\beta^2(u_5))\rho(\beta(u_3)\!,\!\beta(u_2))(\alpha_M(x_1))\\
&&-\rho(\beta^2(u_5)\!,\!\beta^2(u_3))\rho(\beta(u_4)\!,\!\beta(u_2))(\alpha_M(x_1))\!-\!\rho(\beta^2(u_3)\!,\!\beta^2(u_4))\rho(\beta(u_5)\!,\!\beta(u_2))(\alpha_M(x_1))\\
&=&[\beta^2(u_4),\beta^2(u_5),[\beta(u_1),\beta(u_2),\alpha(u_3)]]+\rho(\beta^2(u_4),\beta^2(u_5))\big(\rho(\beta(u_1),\beta(u_2))(\alpha_M(x_3))\\
&&-\rho(\beta(u_1),\beta(u_3))(\alpha_M(x_2))+\rho(\beta(u_2),\beta(u_3))(\alpha_M(x_1))\big)\\
&&-\rho(\beta^2(u_4),\alpha^{-1}\beta([\beta(u_1),\beta(u_2),\alpha(u_3)]))(\alpha_M\beta_M(x_5))\\
&&+\rho(\beta^2(u_5),\alpha^{-1}\beta([\beta(u_1),\beta(u_2),\alpha(u_3)]))(\alpha_M\beta_M(x_4))\\
&&-[\beta^2(u_3),\beta^2(u_5),[\beta(u_1),\beta(u_2),\alpha(u_4)]]-\rho(\beta^2(u_3),\beta^2(u_5))\big(\rho(\beta(u_1),\beta(u_2))(\alpha_M(x_4))\\
&&-\rho(\beta(u_1),\beta(u_4))(\alpha_M(x_2))+\rho(\beta(u_2),\beta(u_4))(\alpha_M(x_1))\big)\\
&&+\rho(\beta^2(u_3),\alpha^{-1}\beta([\beta(u_1),\beta(u_2),\alpha(u_4)]))(\alpha_M\beta_M(x_5))\\
&&-\rho(\beta^2(u_5),\alpha^{-1}\beta([\beta(u_1),\beta(u_2),\alpha(u_4)]))(\alpha_M\beta_M(x_3))\\
&&[\beta^2(u_3),\beta^2(u_4),[\beta(u_1),\beta(u_2),\alpha(u_5)]]+\rho(\beta^2(u_3),\beta^2(u_4))\big(\rho(\beta(u_1),\beta(u_2))(\alpha_M(x_5))\\
&&-\rho(\beta(u_1),\beta(u_5))(\alpha_M(x_2))+\rho(\beta(u_2),\beta(u_5))(\alpha_M(x_1))\big)\\
&&-\rho(\beta^2(u_3),\alpha^{-1}\beta([\beta(u_1),\beta(u_2),\alpha(u_4)]))(\alpha_M\beta_M(x_4))\\
&&+\rho(\beta^2(u_4),\alpha^{-1}\beta([\beta(u_1),\beta(u_2),\alpha(u_5)]))(\alpha_M\beta_M(x_3))\\
&=&[(\beta+\beta_M)^2(u_4+x_4),(\beta+\beta_M)^2(u_5+x_5),[(\beta+\beta_M)(u_1+x_1),(\beta+\beta_M)(u_2+x_2),\\
&&\,(\alpha+\alpha_M)(u_3+x_3)]_\rho]_\rho\!-\![(\beta\!+\!\beta_M)^2(u_3\!+\!x_3),(\beta\!+\!\beta_M)^2(u_5\!+\!x_5),[(\beta\!+\!\beta_M)(u_1+x_1),\\
&&\,(\beta+\beta_M)(u_2+x_2),(\alpha+\alpha_M)(u_4+x_4)]_\rho]_\rho\!+\![(\beta\!+\!\beta_M)^2(u_3\!+\!x_3),(\beta+\beta_M)^2(u_4+x_4),\\
&&\,[(\beta+\beta_M)(u_1+x_1),(\beta+\beta_M)(u_2+x_2),(\alpha+\alpha_M)(u_5+x_5)]_\rho]_\rho.
\end{eqnarray*}

Thus, $L\ltimes M:=(L\oplus M, [\cdot,\cdot,\cdot]_\rho,\alpha+ \alpha_M,\beta+\beta_M)$ is a $3$-Bihom-Lie algebra.
\end{proof}

\begin{defn}
Let $(L, [\cdot, \cdot, \cdot], \alpha,\beta)$ be a $3$-Bihom-Lie algebra and  $(M, \rho, \alpha_M, \beta_M)$ be a
representation of $L$. If $\theta:L\times L\times L\rightarrow M$ is a $3$-linear map and satisfies
\begin{enumerate}[(1)]
\item $\alpha_M\theta(x_1,x_2,x_3)=\theta(\alpha(x_1),\alpha(x_2),\alpha(x_3)),$
\item $\beta_M\theta(x_1,x_2,x_3)=\theta(\beta(x_1),\beta(x_2),\beta(x_3)),$
\item $\theta(\beta(x_1),\beta(x_2),\alpha(x_3))=-\theta(\beta(x_2),\beta(x_1),\alpha(x_3))=-\theta(\beta(x_1),\beta(x_3),\alpha(x_2)),$
\item $\quad\theta(\beta^2(x_1),\beta^2(x_2),[\beta(x_3),\beta(x_4),\alpha(x_5)])+\rho(\beta^2(x_1),\beta^2(x_2))\theta(\beta(x_3),\beta(x_4),\alpha(x_5))\\
=\theta(\beta^2(x_4),\beta^2(x_5),[\beta(x_1),\beta(x_2),\alpha(x_3)])+\rho(\beta^2(x_4),\beta^2(x_5))\theta(\beta(x_1),\beta(x_2),\alpha(x_3))\\
~~~ -\theta(\beta^2(x_3),\beta^2(x_5),[\beta(x_1),\beta(x_2),\alpha(x_4)])-\rho(\beta^2(x_3),\beta^2(x_5))\theta(\beta(x_1),\beta(x_2),\alpha(x_4))\\
~~~ +\theta(\beta^2(x_3),\beta^2(x_4),[\beta(x_1),\beta(x_2),\alpha(x_5)])+\rho(\beta^2(x_3),\beta^2(x_4))\theta(\beta(x_1),\beta(x_2),\alpha(x_5)),$
\end{enumerate}
where $ x_1,x_2,x_3,x_4,x_5 \in L$.
Then $\theta$ is called a $3$-cocycle associated with $\rho$.
\end{defn}

\begin{prop}\label{prop1}
Let $(L, [\cdot, \cdot, \cdot], \alpha,\beta)$ be a $3$-Bihom-Lie algebra and  $(M, \rho, \alpha_M, \beta_M)$ be a
representation of $L$. Assume that the maps $\alpha$ and $\beta_M$ are surjective. If $\theta$ is a $3$-cocycle associated with $\rho$. Then $(L\oplus M, [\cdot,\cdot,\cdot]_\theta,\alpha+ \alpha_M,\beta+\beta_M)$ is a $3$-Bihom-Lie algebra, where
$\alpha+\alpha_M,\beta+\beta_M:L\oplus M\rightarrow L\oplus M$ are defined by $(\alpha+\alpha_M)(u+x)=\alpha(u)+\alpha_M(x)$ and $(\beta+\beta_M)(u+x)=\beta(u)+\beta_M(x)$, and the bracket $[\cdot,\cdot,\cdot]_\theta$ is
defined by
\begin{eqnarray*}
&&[u+x, v+y, w+z]_{\theta}\\
&=&\![u, v,z]\!+\!\theta(u,v,w)\!+\!\rho(u,v)(z)\!-\!\rho(u, \alpha^{\!-\!1}\beta(w))(\alpha_M\beta_M^{\!-\!1}(y))\!+\!\rho(v, \alpha^{\!-\!1}\beta(w))(\alpha_M\beta_M^{\!-\!1}(x)),
\end{eqnarray*}
for all $u, v, w \in L$ and $x,y,z\in M$. $(L\oplus M, [\cdot,\cdot,\cdot]_\theta,\alpha+ \alpha_M,\beta+\beta_M)$ is called the $T_\theta$-extension of $(L, [\cdot, \cdot, \cdot], \alpha,\beta)$ by $M$, denoted by $T_\theta{(L)}$.
\end{prop}
\begin{proof}
The proof is similar to Proposition \ref{product}.
\end{proof}

\begin{prop}
Let $(L, [\cdot, \cdot, \cdot], \alpha,\beta)$ be a $3$-Bihom-Lie algebra and  $(M, \rho, \alpha_M, \beta_M)$ be a
representation of $L$. Assume that the maps $\alpha$ and $\beta$ are surjective. $f:L\rightarrow M$ is a linear map such that $f\circ\alpha=\alpha_M\circ f$ and $f\circ\beta=\beta_M\circ f$. Then the $3$-linear map $\theta_f:L\times L\times L\rightarrow M$ given by $$\theta_f(x,y,z)=f([x,y,z])\!-\!\rho(x,y)f(z)\!+\!\rho(x,\alpha^{-1}\beta(z))f(\alpha\beta^{-1}(y))\!-\!\rho(y,\alpha^{-1}\beta(z))f(\alpha\beta^{-1}(x)),$$  for all $x,y,z\in L$, is a $3$-cocycle associated with $\rho$.
\end{prop}
\begin{proof}
$\forall\, x,y,z\in L$, we have
\begin{eqnarray*}
&&\theta_f(\alpha(x),\alpha(y),\alpha(z))\\
&=&f([\alpha(x),\alpha(y),\alpha(z)])-\rho(\alpha(x),\alpha(y))f(\alpha(z))+\rho(\alpha(x),\alpha\alpha^{-1}\beta(z))f(\alpha\alpha\beta^{-1}(y))\\
&&-\rho(\alpha(y),\alpha\alpha^{-1}\beta(z))f(\alpha\alpha\beta^{-1}(x))\\
&=&\alpha_Mf([x,y,z])-\alpha_M\rho(x,y)f(z)+\alpha_M\rho(x,\alpha^{-1}\beta(z))f(\alpha\beta^{-1}(y))\\
&&-\alpha_M\rho(y,\alpha^{-1}\beta(z))f(\alpha\beta^{-1}(x))\\
&=&\alpha_M\theta_f(x,y,z).
\end{eqnarray*}
We also have $\theta_f(\beta(x),\beta(y),\beta(z))=\beta_M\theta_f(x,y,z)$.

Next, a calculation shows that
\begin{eqnarray*}
&&\theta_f(\beta(x),\beta(y),\alpha(z))\\
&=&f([\beta(x),\beta(y),\alpha(z)])-\rho(\beta(x),\beta(y))f(\alpha(z))+\rho(\beta(x),\alpha\alpha^{-1}\beta(z))f(\alpha\beta^{-1}\beta(y))\\
&&-\rho(\beta(y),\alpha\alpha^{-1}\beta(z))f(\alpha\beta^{-1}\beta(x))\\
&=&-f([\beta(y),\beta(x),\alpha(z)])+\rho(\beta(y),\beta(x))f(\alpha(z))-\rho(\beta(y),\alpha\alpha^{-1}\beta(z))f(\alpha\beta^{-1}\beta(x))\\
&&+\rho(\beta(x),\alpha\alpha^{-1}\beta(z))f(\alpha\beta^{-1}\beta(y))\\
&=&-\theta_f(\beta(y),\beta(x),\alpha(z)).
\end{eqnarray*}
Similarly, we can get $\theta_f(\beta(x),\beta(y),\alpha(z))=-\theta_f(\beta(x),\beta(z),\alpha(y))$.

Finally, $\forall\, x_1,x_2,x_3,x_4,x_5 \in L$, we have
\begin{eqnarray*}
&&\theta_f(\beta^2(x_1),\beta^2(x_2),[\beta(x_3),\beta(x_4),\alpha(x_5)])+\rho(\beta^2(x_1),\beta^2(x_2))\theta_f(\beta(x_3),\beta(x_4),\alpha(x_5))\\
&=&f([\beta^2(x_1),\beta^2(x_2),[\beta(x_3),\beta(x_4),\alpha(x_5)]])-\rho(\beta^2(x_1),\beta^2(x_2))f([\beta(x_3),\beta(x_4),\alpha(x_5)])\\
&&+\rho(\beta^2(x_1),\alpha^{-1}\beta([\beta(x_3),\beta(x_4),\alpha(x_5)]))f\alpha\beta(x_2)\\
&&-\rho(\beta^2(x_2),\alpha^{-1}\beta([\beta(x_3),\beta(x_4),\alpha(x_5)]))f\alpha\beta(x_1)\\
&&+\rho(\beta^2(x_1),\beta^2(x_2))\big(f([\beta(x_3),\beta(x_4),\alpha(x_5)])-\rho(\beta(x_3),\beta(x_4))f\alpha(x_5)\\
&&+\rho(\beta(x_3),\beta(x_5))f\alpha(x_4)-\rho(\beta(x_4),\beta(x_5))f\alpha(x_3)\big)\\
&=&f([\beta^2(x_1),\beta^2(x_2),[\beta(x_3),\beta(x_4),\alpha(x_5)]])\\
&&+\rho(\beta^2(x_1),\alpha^{-1}\beta([\beta(x_3),\beta(x_4),\alpha(x_5)]))f\alpha\beta(x_2)\\
&&-\rho(\beta^2(x_2),\alpha^{-1}\beta([\beta(x_3),\beta(x_4),\alpha(x_5)]))f\alpha\beta(x_1)\\
&&+\rho(\beta^2(x_1),\beta^2(x_2))\big(-\rho(\beta(x_3),\beta(x_4))f\alpha(x_5)+\rho(\beta(x_3),\beta(x_5))f\alpha(x_4)\\
&&-\rho(\beta(x_4),\beta(x_5))f\alpha(x_3)\big)\\
&=&f([\beta^2(x_4),\beta^2(x_5),[\beta(x_1),\beta(x_2),\alpha(x_3)]])-f([\beta^2(x_3),\beta^2(x_5),[\beta(x_1),\beta(x_2),\alpha(x_4)]])\\
&&+f([\beta^2(x_3),\beta^2(x_4),[\beta(x_1),\beta(x_2),\alpha(x_5)]])-\rho(\beta^2(x_4),\beta^2(x_5))\rho(\beta(x_3),\beta(x_1))f\alpha(x_2)\\
&&-\rho(\beta^2(x_5),\beta^2(x_3))\rho(\beta(x_4),\beta(x_1))f\alpha(x_2)-\rho(\beta^2(x_3),\beta^2(x_4))\rho(\beta(x_5),\beta(x_1))f\alpha(x_2)\\
&&+\rho(\beta^2(x_4),\beta^2(x_5))\rho(\beta(x_3),\beta(x_2))f\alpha(x_1)+\rho(\beta^2(x_5),\beta^2(x_3))\rho(\beta(x_4),\beta(x_2))f\alpha(x_1)\\
&&+\rho(\beta^2(x_3),\beta^2(x_4))\rho(\beta(x_5),\beta(x_2))f\alpha(x_1)-\rho(\beta^2(x_3),\beta^2(x_4))\rho(\beta(x_1),\beta(x_2))f\alpha(x_5)\\
&&-\rho(\alpha^{-1}\beta([\beta(x_1),\beta(x_2),\alpha(x_3)]),\beta^2(x_4))f\alpha\beta(x_5)\\
&&-\rho(\beta^2(x_3)\!,\!\alpha^{-1}\beta([\beta(x_1)\!,\!\beta(x_2)\!,\!\alpha(x_4)])\!)\!f\alpha\beta(x_5)\!+\!\rho(\beta^2(x_3)\!,\!\beta^2(x_5)\!)\!\rho(\beta(x_1)\!,\!\beta(x_2))f\alpha(x_4)\\
&&+\rho(\alpha^{-1}\beta([\beta(x_1),\beta(x_2),\alpha(x_3)]),\beta^2(x_5))f\alpha\beta(x_4)\\
&&+\rho(\beta^2(x_3),\alpha^{-1}\beta([\beta(x_1),\beta(x_2),\alpha(x_5)]))f\alpha\beta(x_4)\\
&&-\rho(\beta^2(x_4),\beta^2(x_5))\rho(\beta(x_1),\beta(x_2))f\alpha(x_3)\\
&&-\rho(\alpha^{-1}\beta([\beta(x_1),\beta(x_2),\alpha(x_4)]),\beta^2(x_5))f\alpha\beta(x_3)\\
&&-\rho(\beta^2(x_4),\alpha^{-1}\beta([\beta(x_1),\beta(x_2),\alpha(x_5)]))f\alpha\beta(x_3)\\
&=&f([\beta^2(x_4)\!,\!\beta^2(x_5)\!,\![\beta(x_1)\!,\!\beta(x_2)\!,\!\alpha(x_3)]])\!+\!\rho(\beta^2(x_4),\alpha^{-1}\beta([\beta(x_1),\beta(x_2),\alpha(x_3)]))f\alpha\beta(x_5)\\
&&-\rho(\beta^2(x_5),\alpha^{-1}\beta([\beta(x_1),\beta(x_2),\alpha(x_3)]))f\alpha\beta(x_4)\\
&&+\rho(\beta^2(x_4),\beta^2(x_5))\big(-\rho(\beta(x_1),\beta(x_2))f\alpha(x_3)+\rho(\beta(x_1),\beta(x_3))f\alpha(x_2)\\
&&-\rho(\beta(x_2),\beta(x_3))f\alpha(x_1)\big)\\
&&-f([\beta^2(x_3)\!,\!\beta^2(x_5)\!,\![\beta(x_1)\!,\!\beta(x_2)\!,\!\alpha(x_4)]])\!-\!\rho(\beta^2(x_3)\!,\!\alpha^{-1}\beta([\beta(x_1)\!,\!\beta(x_2)\!,\!\alpha(x_4)]))f\alpha\beta(x_5)\\
&&+\rho(\beta^2(x_5),\alpha^{-1}\beta([\beta(x_1),\beta(x_2),\alpha(x_4)]))f\alpha\beta(x_3)\\
&&-\rho(\beta^2(x_3),\beta^2(x_5))\big(-\rho(\beta(x_1),\beta(x_2))f\alpha(x_4)+\rho(\beta(x_1),\beta(x_4))f\alpha(x_2)\\
&&-\rho(\beta(x_2),\beta(x_4))f\alpha(x_1)\big)\\
&&+f([\beta^2(x_3)\!,\!\beta^2(x_4)\!,\![\beta(x_1)\!,\!\beta(x_2)\!,\!\alpha(x_5)]])\!+\!\rho(\beta^2(x_3)\!,\!\alpha^{-1}\beta([\beta(x_1)\!,\!\beta(x_2)\!,\!\alpha(x_5)]))f\alpha\beta(x_4)\\
&&-\rho(\beta^2(x_4),\alpha^{-1}\beta([\beta(x_1),\beta(x_2),\alpha(x_5)]))f\alpha\beta(x_3)\\
&&+\rho(\beta^2(x_3),\beta^2(x_4))\big(-\rho(\beta(x_1),\beta(x_2))f\alpha(x_5)+\rho(\beta(x_1),\beta(x_5))f\alpha(x_2)\\
&&-\rho(\beta(x_2),\beta(x_5))f\alpha(x_1)\big)\\
&=&+\theta_f(\beta^2(x_4),\beta^2(x_5),[\beta(x_1),\beta(x_2),\alpha(x_3)])+\rho(\beta^2(x_4),\beta^2(x_5))\theta_f(\beta(x_1),\beta(x_2),\alpha(x_3))\\
&&-\theta_f(\beta^2(x_3),\beta^2(x_5),[\beta(x_1),\beta(x_2),\alpha(x_4)])-\rho(\beta^2(x_3),\beta^2(x_5))\theta_f(\beta(x_1),\beta(x_2),\alpha(x_4))\\
&&+\theta_f(\beta^2(x_3),\beta^2(x_4),[\beta(x_1),\beta(x_2),\alpha(x_5)])+\rho(\beta^2(x_3),\beta^2(x_4))\theta_f(\beta(x_1),\beta(x_2),\alpha(x_5)).
\end{eqnarray*}

It follows that $\theta_f$ is a $3$-cocycle associated with $\rho$.
\end{proof}

\begin{cor}
Under the above notations, $\theta+\theta_f$ is a $3$-cocycle associated with $\rho$.
\end{cor}

\begin{prop}
Under the above notations, $\sigma:T_\theta(L)\rightarrow T_{\theta+\theta_f}(L)$ is a isomorphism of $3$-Bihom-Lie algebras, where $\sigma(v+x)=v+f(v)+x,~ \forall\, v\in L, x \in M$.
\end{prop}
\begin{proof}
It is clear that $\sigma$ is a bijection. Next, for every $v_i\in L, x_i\in M, i=1,2,3$, we have $\sigma\circ(\alpha+\alpha_M)(v_1+x_1)=\sigma(\alpha(v_1)+\alpha_M(x_1))=\alpha(v_1)+f\alpha(v_1)+\alpha_M(x_1)=\alpha(v_1)+\alpha_Mf(v_1)+\alpha_M(x_1)=(\alpha+\alpha_M)(v_1+f(v_1)+x_1)=(\alpha+\alpha_M)\circ\sigma(v_1+x_1)$, i.e. $\sigma\circ(\alpha+\alpha_M)=(\alpha+\alpha_M)\circ\sigma$. Similarly, $\sigma\circ(\beta+\beta_M)=(\beta+\beta_M)\circ\sigma$.

Then we can obtain
\begin{eqnarray*}
&&[\sigma(v_1+x_1),\sigma(v_2+x_2),\sigma(v_3+x_3)]_{\theta+\theta_f}\\
&=&[v_1+f(v_1)+x_1,v_2+f(v_2)+x_2,v_3+f(v_3)+x_3]_{\theta+\theta_f}\\
&=&[v_1,v_2,v_3]+(\theta+\theta_f)(v_1,v_2,v_3)+\rho(v_1,v_2)(f(v_3)+x_3)\\
&&-\rho(v_1,\alpha^{-1}\beta(v_3))\alpha_M\beta_M^{-1}(f(v_2)+x_2)+\rho(v_2,\alpha^{-1}\beta(v_3))\alpha_M\beta_M^{-1}(f(v_1)+x_1)\\
&=&[v_1,v_2,v_3]+\theta(v_1,v_2,v_3)+f([v_1,v_2,v_3])-\rho(v_1,v_2)f(v_3)+\rho(v_1,\alpha^{-1}\beta(v_3))f\alpha\beta^{-1}(v_2)\\
&&-\rho(v_2,\alpha^{-1}\beta(v_3))f\alpha\beta^{-1}(v_1)+\rho(v_1,v_2)f(v_3)-\rho(v_1,\alpha^{-1}\beta(v_3))\alpha_M\beta^{-1}_Mf(v_2)\\
&&+\rho(v_2,\alpha^{-1}\beta(v_3))\alpha_M\beta^{-1}_Mf(v_1)+\rho(v_1,v_2)(x_3)-\rho(v_1,\alpha^{-1}\beta(v_3))\alpha_M\beta^{-1}_M(x_2)\\
&&+\rho(v_2,\alpha^{-1}\beta(v_3))\alpha_M\beta^{-1}_M(x_1)\\
&=&[v_1,v_2,v_3]+f([v_1,v_2,v_3])+\theta(v_1,v_2,v_3)+\rho(v_1,v_2)(x_3)-\rho(v_1,\alpha^{-1}\beta(v_3))\alpha_M\beta^{-1}_M(x_2)\\
&&+\rho(v_2,\alpha^{-1}\beta(v_3))\alpha_M\beta^{-1}_M(x_1)\\
&=&\sigma\big([v_1,v_2,v_3]+\theta(v_1,v_2,v_3)+\rho(v_1,v_2)(x_3)-\rho(v_1,\alpha^{-1}\beta(v_3))\alpha_M\beta^{-1}_M(x_2)\\
&&+\rho(v_2,\alpha^{-1}\beta(v_3))\alpha_M\beta^{-1}_M(x_1)\big)\\
&=&\sigma([v_1+x_1,v_2+x_2,v_3+x_3]_\theta).
\end{eqnarray*}

That shows $\sigma$ is a isomorphism.
\end{proof}

\section{$T_\theta^*$-extensions of $3$-Bihom-Lie algebras}
The method of $T_\theta^*$-extension was introduced in \cite{MB} and has already been used for $3$-Lie algebras in \cite{RWYZ} and $3$-hom-Lie algebras in \cite{YLY}. Now we will
generalize it to $3$-Bihom-Lie algebras.
\begin{defn}
Let $(L,[\cdot,\cdot,\cdot],\alpha,\beta)$ be a $3$-Bihom-Lie algebra. A bilinear form  $f$ on $L$ is said to be nondegenerate if
$$L^\perp=\{x\in L~|~f(x,y)=0, ~\forall\, y\in L\}=0;$$
$\alpha\beta$-invariant if for all $x_1,x_2,x_3,x_4\in L$,
$$f([\beta(x_1),\beta(x_2),\alpha(x_3)],\alpha(x_4))=-f(\alpha(x_3),[\beta(x_1),\beta(x_2),\alpha(x_4)]);$$
symmetric if
$$f(x,y)=f(y,x).$$
A subspace $I$ of $L$ is called isotropic if $I\subseteq I^\bot$.
\end{defn}

\begin{defn}
Let $(L,[\cdot,\cdot,\cdot],\alpha,\beta)$ be a $3$-Bihom-Lie algebra over a field $\mathbb{K}$. If $L$ admits a nondegenerate, $\alpha\beta$-invariant and symmetric bilinear form $f$ such that $\alpha$, $\beta$ are $f$-symmetric $($i.e. $f(\alpha(x),y)=f(x,\alpha(y)),~ f(\beta(x),y)=f(x,\beta(y))$$)$, then we call $(L,f,\alpha,\beta)$ a quadratic $3$-Bihom-Lie algebra.

Let $(L^{'},[\cdot,\cdot,\cdot]',\alpha',\beta')$ be another $3$-Bihom-Lie algebra. Two quadratic $3$-Bihom-Lie algebras $(L,f,\alpha,\beta)$ and
$(L^{'},f',\alpha',\beta')$ are said to be isometric if there exists a algebra isomorphism $\phi: L\rightarrow L^{'}$ such that
$f(x, y)=f'(\phi(x), \phi(y)),~ \forall\, x, y\in L$.
\end{defn}

\begin{thm}
Let $(L,[\cdot,\cdot,\cdot],\alpha,\beta)$ be a $3$-Bihom-Lie algebra and $(M, \rho, \alpha_M, \beta_M)$ be a representation of $L$. Let us consider $M^*$ the dual space of $M$ and $\tilde\alpha_M,\tilde\beta_M:M^*\rightarrow M^*$ two homomorphisms defined by $
\tilde\alpha_M(f)=f\circ\alpha_M,\tilde\beta_M(f)=f\circ\beta_M,~ \forall\, f\in M^*$. Then the skewsymmetry linear map $\tilde\rho:L\times L\rightarrow \mathrm{End}(M^{*})$, defined by $\tilde\rho(x,y)(f)=- f\circ \rho(x,y),~ \forall\, f \in M^*, x, y \in L,$ is a representation of $L$ on $(M^{*},\tilde\rho,\tilde{\alpha}_M,\tilde\beta_M)$ if and only if for every $x,y,u,v\in L$,
\begin{enumerate}[(1)]
\item $\alpha_M\circ \rho(\alpha(x),\alpha(y))=\rho(x,y)\circ \alpha_M,$
\item $\beta_M\circ \rho(\beta(x),\beta(y))=\rho(x,y)\circ \beta_M,$
\item $\quad\rho(x,y)\rho(\alpha\beta(u),\alpha\beta(v))\\
    =\rho(\alpha(u),\alpha(v))\rho(\beta(u),\beta(v))\!-\!\beta_M\rho([\beta(u),\beta(v),x],\beta(y))\!-\!\beta_M\rho(\beta(x),[\beta(u),\beta(v),y]),$
\item $\quad\beta_M\rho([\beta(u),\beta(v),x],\beta(y))\\
    =-\rho(\alpha(u),y)\rho(\alpha\beta(v),\beta(x))\!-\!\rho(\alpha(v),y)\rho(\beta(x),\alpha\beta(u))\!-\!\rho(x,y)\rho(\alpha\beta(u),\alpha\beta(v)),$
\end{enumerate}
\end{thm}

\begin{proof}Let $f \in M^*$, $x,y,u,v \in L$.
First, we have
$$
(\tilde\rho(\alpha(u),\alpha(v))\circ \tilde{\alpha}_M)(f)=- \tilde{\alpha}_M(f)\circ \rho(\alpha(u),\alpha(v))=-f\circ \alpha_M\circ \rho(\alpha(u),\alpha(v))$$ and
$\tilde{\alpha}_M\circ\tilde\rho(u,v)(f)=-\tilde{\alpha}_M(f\circ\rho(u,v))=-f\circ\rho(u,v)\circ\alpha_M,$ which implies $$\tilde\rho(\alpha(u),\alpha(v))\circ \tilde{\alpha}_M=\tilde{\alpha}_M\circ\tilde\rho(u,v)\Leftrightarrow \alpha_M\circ \rho(\alpha(u),\alpha(v))=\rho(u,v)\circ\alpha_M.$$
Similarly, $\tilde\rho(\beta(u),\beta(v))\circ \tilde{\beta}_M=\tilde{\beta}_M\circ\tilde\rho(u,v)\Leftrightarrow \beta_M\circ \rho(\beta(u),\beta(v))=\rho(u,v)\circ\beta_M.$

Then we can get
$$
\tilde\rho(\alpha\beta(u),\alpha\beta(v))\circ\tilde\rho(x,y)(f)=-\tilde\rho(\alpha\beta(u),\alpha\beta(v))(f\rho(x,y))=f\rho(x,y)\rho(\alpha\beta(u),\alpha\beta(v))
$$ and
\begin{eqnarray*}
&&\big(\tilde\rho(\beta(x),\beta(y))\circ\tilde\rho(\alpha(u),\alpha(v))+\tilde\rho([\beta(u),\beta(v),x],\beta(y))\circ\tilde\beta_M\\
&&\,+\tilde\rho(\beta(x),[\beta(u),\beta(v),y])\circ\tilde\beta_M\big)(f)\\
&=&\!f\rho(\alpha(u),\alpha(v)\!)\!\rho(\beta(x),\beta(y)\!)\!\!-\!f\beta_M\rho([\beta(u),\beta(v),x],\beta(y)\!)\!\!-\!f\beta_M\rho(\beta(x),[\beta(u),\beta(v),y]),
\end{eqnarray*}
which implies
\begin{eqnarray*}
&&\tilde\rho(\alpha\beta(u),\alpha\beta(v))\circ\tilde\rho(x,y)\\
&=&\tilde\rho(\beta(x),\!\beta(y)\!)\!\circ\!\tilde\rho(\alpha(u),\!\alpha(v)\!)\!+\!\tilde\rho([\beta(u),\!\beta(v),\!x],\!\beta(y)\!)\!\circ\!\tilde\beta_M
\!+\!\tilde\rho(\beta(x),[\beta(u),\beta(v),y])\!\circ\!\tilde\beta_M
\end{eqnarray*}
if and only if
\begin{eqnarray*}
&&\rho(x,y)\rho(\alpha\beta(u),\alpha\beta(v))\\
&=&\rho(\alpha(u),\alpha(v))\rho(\beta(u),\beta(v))\!-\!\beta_M\rho([\beta(u),\beta(v),x],\beta(y))\!-\!\beta_M\rho(\beta(x),[\beta(u),\beta(v),y]).
\end{eqnarray*}
In the same way,
\begin{eqnarray*}
&&\tilde\rho([\beta(u),\beta(v),x],\beta(y))\tilde\beta_M\\
&=&\tilde\rho(\alpha\beta(v),\beta(x)\!)\!\circ\!\tilde\rho(\alpha(u),y\!)\!\!+\!\tilde\rho(\beta(x),\alpha\beta(u)\!)\!\circ\!\tilde\rho(\alpha(v),y)
\!+\!\tilde\rho(\alpha\beta(u),\alpha\beta(v))\!\circ\!\tilde\rho(x,y)
\end{eqnarray*}
if and only if
\begin{eqnarray*}
&&\beta_M\rho([\beta(u),\beta(v),x],\beta(y))\\
&=&\!-\!\rho(\alpha(u),y)\rho(\alpha\beta(v),\beta(x))\!-\!\rho(\alpha(v),y)\rho(\beta(x),\alpha\beta(u))\!-\!\rho(x,y)\rho(\alpha\beta(u),\alpha\beta(v)).
\end{eqnarray*}

That shows the theorem holds.
\end{proof}

\begin{cor}
Let $\ad$ be the adjoint representation of a $3$-Bihom-Lie algebra $(L,[\cdot,\cdot,\cdot],\\
\alpha,\beta)$. Let us consider the bilinear map $\ad^*:L\times L\rightarrow \mathrm{End}(L^{*})$ defined by $$\ad^*(x,y)(f)=-f\circ \ad(x,y),~\forall\, x,y\in L.$$ Then $\ad^*$ is a representation of $L$ on $(L^{*},\ad^*,\tilde{\alpha},\tilde\beta)$ if and only if
\begin{enumerate}[(1)]
\item $\alpha\circ \ad(\alpha(x),\alpha(y))=\ad(x,y)\circ \alpha,$
\item $\beta\circ \ad(\beta(x),\beta(y))=\ad(x,y)\circ \beta,$
\item $\quad\ad(x,y)\ad(\alpha\beta(u),\alpha\beta(v))\\
    =\ad(\alpha(u),\alpha(v))\ad(\beta(u),\beta(v))\!-\!\beta\ad([\beta(u),\beta(v),x],\beta(y))\!-\!\beta\ad(\beta(x),[\beta(u),\beta(v),y]),$
\item $\quad\beta\ad([\beta(u),\beta(v),x],\beta(y))\\
    =\!-\ad(\alpha(u),y)\ad(\alpha\beta(v),\beta(x)\!)-\ad(\alpha(v),y)\ad(\beta(x),\alpha\beta(u)\!)-\ad(x,y)\ad(\alpha\beta(u),\alpha\beta(v)\!),$
\end{enumerate}
We call the representation $\ad^*$ the coadjoint representation of $L$.
\end{cor}
Under the above notations, assume that the coadjoint representation $\ad^*$
exists and $\alpha$, $\beta$ are bijective. From Proposition \ref{prop1}, $(L\oplus L^{*},[\cdot,\cdot,\cdot]_\theta,\alpha+\tilde{\alpha},\beta+\tilde{\beta})$ is a $3$-Bihom-Lie algebra by a $3$-cocycle $\theta:L\times L\times L\rightarrow L^*$ associated with $\ad^*$.

\begin{defn}
Let $L$ be a $3$-Bihom-Lie algebra over a field $\mathbb{K}$. We inductively define a derived series
$$(L^{(n)})_{n\geq 0}: L^{(0)}=L,\ L^{(n+1)}=[L^{(n)},L^{(n)},L]$$
and a central descending series
$$(L^{n})_{n\geq 0}: L^{0}=L,\ L^{n+1}=[L^{n},L,L].$$

$L$ is called solvable and nilpotent $($of length $k$$)$ if and only if there is a $($smallest$)$ integer $k$ such that $L^{(k)}=0$ and $L^{k}=0$, respectively.
\end{defn}

\begin{thm}
Let $(L,[\cdot,\cdot,\cdot],\alpha,\beta)$ be a $3$-Bihom-Lie algebra over a field $\mathbb{K}$.
\begin{enumerate}[(1)]
   \item  If $L$ is solvable, then $(L\oplus L^{*},[\cdot,\cdot,\cdot]_\theta,\alpha+\tilde{\alpha},\beta+\tilde{\beta})$ is solvable.
   \item  If $L$ is nilpotent, then $(L\oplus L^{*},[\cdot,\cdot,\cdot]_\theta,\alpha+\tilde{\alpha},\beta+\tilde{\beta})$ is nilpotent.
\end{enumerate}
\end{thm}
\begin{proof}
(1) We suppose that $L$ is solvable of length $s$, i.e. $L^{(s)}=[L^{(s-1)},L^{(s-1)},L]=0.$ We claim that $(L\oplus L^*)^{(k)}\subseteq L^{(k)}+L^*$, which we prove by induction on $k$. The case $k=1$, by Proposition \ref{prop1}, we have
\begin{eqnarray*}
(L\oplus L^*)^{(1)}&=&[L\oplus L^*,L\oplus L^*,L\oplus L^*]_\theta\\
&=&[L,L,L]_\theta+[L,L,L^*]_\theta+[L,L^*,L]_\theta+[L^*,L,L]_\theta\\
&=&[L,L,L]+\theta(L,L,L)+[L,L,L^*]_\theta+[L,L^*,L]_\theta+[L^*,L,L]_\theta\\
&\subseteq&L^{(1)}+L^*.
\end{eqnarray*}
By induction, $(L\oplus L^*)^{(k-1)}\subseteq L^{(k-1)}+L^*$. So
\begin{eqnarray*}
&&(L\oplus L^*)^{(k)}\\
&=&[(L\oplus L^*)^{(k-1)},(L\oplus L^*)^{(k-1)},L\oplus L^*]_\theta\\
&\subseteq&[L^{(k-1)}+L^*,L^{(k-1)}+L^*,L\oplus L^*]_\theta\\
&=&[L^{(k-1)},L^{(k-1)},L]+\theta(L^{(k-1)},L^{(k-1)},L)+[L^{(k-1)},L^{(k-1)},L^*]_\theta+[L^{(k-1)},L^*,L]_\theta\\
&&+[L^*,L^{(k-1)},L]_\theta\\
&\subseteq&L^{(k)}+L^*.
\end{eqnarray*}
Therefore
\begin{eqnarray*}
&&(L\oplus L^*)^{(s+1)}\\
&\subseteq&[L^{(s)},L^{(s)},L]+\theta(L^{(s)},L^{(s)},L)+[L^{(s)},L^{(s)},L^*]_\theta+[L^{(s)},L^*,L]_\theta+[L^*,L^{(s)},L]_\theta\\
&=&0.
\end{eqnarray*}
It follows $(L\oplus L^{*},[\cdot,\cdot,\cdot]_\theta,\alpha+\tilde{\alpha},\beta+\tilde{\beta})$ is solvable.

(2) Suppose that $L$ is nilpotent of length $s$. Since $(L\oplus L^{*})^{s}/L^{*}\cong L^{s}$ and $L^{s}=0$, we have
$(L\oplus L^{*})^{s}\subseteq L^{*}$. Let $h\in(L\oplus L^{*})^{s}\subseteq L^{*},~ b\in L,~x_{i}+f_{i},~ y_{i}+g_{i}\in L\oplus L^{*}, ~1\leq i\leq s-1$, we have
\begin{eqnarray*}
&&[[\cdots[h,x_1+f_{1},y_{1}+g_{1}]_\theta,\cdots]_\theta,x_{s-1}+f_{s-1},y_{s-1}+g_{s-1}]_\theta(b)\\
&=&(-1)^{s-1}h\alpha\beta^{-1}\ad(x_1,\!\alpha^{-1}\beta(y_1)\!)\alpha\beta^{-1}\ad(x_2,\!\alpha^{-1}\beta(y_2)\!)\!\cdots\!\alpha\beta^{-1}\ad(x_{s-1},\!\alpha^{-1}\beta(y_{s-1})\!)\!(b)\\
&=&(-1)^{s-1}h\alpha\beta^{-1}([x_1,\alpha^{-1}\beta(y_1),\alpha^{-1}\beta[x_2,\alpha^{-1}\beta(y_2),\!\cdots\!,\alpha\beta^{-1}[x_{s-1},\alpha^{-1}\beta(y_{s-1}),b]\cdots]])\\
&\in &h(L^{s})=0.
\end{eqnarray*}
Thus $(L\oplus L^{*},[\cdot,\cdot,\cdot]_\theta,\alpha+\tilde{\alpha},\beta+\tilde{\beta})$ is nilpotent.
\end{proof}

Now we consider the following symmetric bilinear form $q_{L}$ on $L\oplus L^{*}$,
$$q_{L}(x+f,y+g)=f(y)+g(x),~ \forall\, x+f, y+g\in L\oplus L^*.$$
Obviously, $q_{L}$ is  nondegenerate. In fact, if $x+f$ is orthogonal to  all elements $y+g$ of $L\oplus L^{*}$, then $f(y)=0$ and $g(x)=0$,  which implies that $x=0$ and $f=0$.

\begin{lem}\label{lemma3.2}
Let $q_L$ be as above. Then the 4-tuple $(L\oplus L^{*},q_L,\alpha+\tilde{\alpha},\beta+\tilde{\beta})$ is a quadratic $3$-Bihom-Lie algebra if and only if $\theta$ satisfies for all $ x_1,x_2,x_3,x_4\in L$,
\begin{eqnarray*}\label{203}
\theta(\beta(x_1), \beta(x_2), \alpha(x_3))(\alpha(x_4))+\theta(\beta(x_1),\beta(x_3),\alpha(x_3))(\alpha(x_4))=0.
\end{eqnarray*}
\end{lem}
\begin{proof}

Now suppose that  $x_i+f_i\in L\oplus L^*, i=1,2,3,4$, we have
\begin{eqnarray*}
q_{L}((\alpha+\tilde{\alpha})(x_1+f_1),x_2+f_2)&=&q_L(\alpha(x_1)+f_1\circ\alpha,x_2+f_2)\\
&=&f_2\circ\alpha(x_1)+f_1(\alpha(x_2))\\
&=&q_L(x_1+f_1,(\alpha+\tilde{\alpha})(x_2+f_2)).
\end{eqnarray*}
Then $\alpha+\tilde{\alpha}$ is $q_L$-symmetric. In the same way, $\beta+\tilde{\beta}$ is $q_L$-symmetric.

Next, we can obtain
\begin{eqnarray*}
&&q_{L}\big([(\beta+\tilde{\beta})(x_1+f_1),(\beta+\tilde{\beta})(x_2+f_2), (\alpha+\tilde{\alpha})(x_3+f_3)]_\theta, (\alpha+\tilde{\alpha})(x_4+f_4)\big)\\
&&+q_{L}\big((\alpha+\tilde{\alpha})(x_3+f_3),[(\beta+\tilde{\beta})(x_1+f_1), (\beta+\tilde{\beta})(x_2+f_2), (\alpha+\tilde{\alpha})(x_4+f_4)]_\theta\big)\\
&=&q_{L}\big([\beta(x_1)+f_1\circ\beta, \beta(x_2)+f_2\circ\beta, \alpha(x_3)+f_3\circ\alpha]_\theta, \alpha(x_4)+f_4\circ\alpha\big)\\
&&+q_{L}\big(\alpha(x_3)+f_3\circ\alpha, [\beta(x_1)+f_1\circ\beta, \beta(x_2)+f_2\circ\beta, \alpha(x_4)+f_4\circ\alpha]_\theta\big)\\
&=&q_{L}\big([\beta(x_1),\beta(x_2),\alpha(x_3)]+\theta(\beta(x_1),\beta(x_2),\alpha(x_3))+\ad^*(\beta(x_1),\beta(x_2))(f_3\circ\alpha)\\
&&-\ad^*(\beta(x_1),\alpha^{-1}\beta\alpha(x_3))\tilde{\alpha}\tilde{\beta}^{-1}(f_2\circ\beta)+\ad^*(\beta(x_2),\alpha^{-1}\beta\alpha(x_3))\tilde{\alpha}\tilde{\beta}^{-1}(f_1\circ\beta),\\
&&\alpha(x_4)+f_4\circ\alpha\big)+q_{L}\big(\alpha(x_3)+f_3\circ\alpha,[\beta(x_1),\beta(x_2),\alpha(x_4)]+\theta(\beta(x_1),\beta(x_2),\alpha(x_4))\\
&&+\ad^*(\beta(x_1),\beta(x_2))(f_4\circ\alpha)-\ad^*(\beta(x_1),\alpha^{-1}\beta\alpha(x_4))\tilde{\alpha}\tilde{\beta}^{-1}(f_2\circ\beta)\\
&&+\ad^*(\beta(x_2),\alpha^{-1}\beta\alpha(x_4))\tilde{\alpha}\tilde{\beta}^{-1}(f_1\circ\beta)\big)\\
&=&\theta(\beta(x_1),\beta(x_2),\alpha(x_3)\!)(\alpha(x_4)\!)\!-\!f_3\alpha([\beta(x_1),\beta(x_2),\alpha(x_4)])\!+\!f_2\alpha([\beta(x_1),\beta(x_3),\alpha(x_4)])\\
&&\!-\!f_1\alpha([\beta(x_2),\beta(x_3),\alpha(x_4)])\!+\!f_4\alpha([\beta(x_1),\beta(x_2),\alpha(x_3)])\!+\!\theta(\beta(x_1),\beta(x_2),\alpha(x_4)\!)\!(\alpha(x_3)\!)\!\\
&&-f_4\alpha([\beta(x_1),\beta(x_2),\alpha(x_3)])+f_2\alpha([\beta(x_1),\beta(x_4),\alpha(x_3)])-f_1\alpha([\beta(x_2),\beta(x_4),\alpha(x_3)])\\
&&+f_3\alpha([\beta(x_1),\beta(x_2),\alpha(x_4)])\\
&=&\theta(\beta(x_1),\beta(x_2),\alpha(x_3))(\alpha(x_4))+\theta(\beta(x_1),\beta(x_2),\alpha(x_4))(\alpha(x_3)),
\end{eqnarray*}
which implies
\begin{eqnarray*}
&&q_{L}\big([(\beta+\tilde{\beta})(x_1+f_1),(\beta+\tilde{\beta})(x_2+f_2), (\alpha+\tilde{\alpha})(x_3+f_3)]_\theta, (\alpha+\tilde{\alpha})(x_4+f_4)\big)\\
&+&\!\!\!\!q_{L}\big((\alpha+\tilde{\alpha})(x_3+f_3),[(\beta+\tilde{\beta})(x_1+f_1), (\beta+\tilde{\beta})(x_2+f_2), (\alpha+\tilde{\alpha})(x_4+f_4)]_\theta\big)=0
\end{eqnarray*}
if and only if
$\theta(\beta(x_1),\beta(x_2),\alpha(x_3))(\alpha(x_4))+\theta(\beta(x_1),\beta(x_2),\alpha(x_4))(\alpha(x_3))
=0$.

Hence the lemma follows.
\end{proof}

Now, we shall call the quadratic $3$-Bihom-Lie algebra $(L\oplus L^{*},q_L,\alpha+\tilde{\alpha},\beta+\tilde{\beta})$ the $T^*_\theta$-extension of $L$ (by $\theta$) and denote by $T_\theta^*(L)$.

\begin{lem}\label{lemma3.1}
Let $(L,q_L,\alpha,\beta)$ be a $2n$-dimensional quadratic $3$-Bihom-Lie algebra over a field $\mathbb{K}$ $(\ch\mathbb{K}\neq2)$, $\alpha$ be surjective and $I$ be an isotropic $n$-dimensional subspace of $L$. If $I$ is a Bihom-ideal of $(L,[\cdot,\cdot,\cdot],\alpha,\beta)$, then $[\beta(I),\beta(L),\alpha(I)]=0$.
\end{lem}
\begin{proof}
Since dim$I$+dim$I^{\bot}=n+\dim I^{\bot}=2n$ and $I\subseteq I^{\bot}$, we have $I=I^{\bot}$.
If $I$ is a Bihom-ideal of $(L,[\cdot,\cdot],\alpha,\beta)$, then \begin{eqnarray*}
q_L([\beta(I),\beta(L),\alpha(I^{\bot})],\alpha(L))&=&-q_L(\alpha(I^{\bot}),[\beta(I),\beta(L),\alpha(L)])\\
&\subseteq&q_L(\alpha(I^{\bot}),[I,\beta(L),\alpha(L)])\\
&\subseteq&q_L(I^{\bot},I)=0,
\end{eqnarray*}
which implies $[\beta(I),\beta(L),\alpha(I)]=[\beta(I),\beta(L),\alpha(I^{\bot})]\subseteq \alpha(L)^{\bot}=L^{\bot}=0$.
\end{proof}

\begin{thm}
Let $(L,q_L,\alpha,\beta)$ be a quadratic regular $3$-Bihom-Lie algebra of dimensional $2n$ over a field $\mathbb{K}$ $(\ch\mathbb{K}\neq2)$. Then $(L,q_L,\alpha,\beta)$ is isometric to a $T^{*}_\theta$-extension $(T_{\theta}^{*}(B),q_{B},\alpha^{'},\beta^{'})$ if and only if $(L,[\cdot,\cdot,\cdot],\alpha,\beta)$ contains an isotropic Bihom-ideal $I$ of dimension $n$. In particular, $B\cong L/I$.
\end{thm}
\begin{proof}
($\Longrightarrow$) Suppose $\phi:B\oplus B^*\rightarrow L$ is isometric, we have $\phi(B^*)$ is a $n$-dimensional isotropic Bihom-ideal of $L$. In fact, since $\phi$ is isometric, ${\rm dim}B\oplus B^*={\rm dim}L=2n$, which implies ${\rm dim}B^*={\rm dim}\phi(B^*)=n$. And $0=q_B(B^*,B^*)=q_L(\phi(B^*),\phi(B^*))$, we have $\phi(B^*)\subseteq\phi(B^*)^{\bot}$. By $[\phi(B^*),L,L]=[\phi(B^*),\phi(B\oplus B^*),\phi(B\oplus B^*)]=\phi([B^*,B\oplus B^*,B\oplus B^*]_\theta)\subseteq\phi(B^*)$, $\phi(B^*)$ is a Bihom-ideal of $L$. Furthermore, $B\cong B\oplus B^*/B^*\cong L/\phi(B^*)$.

($\Longleftarrow$) Suppose that $I$ is a $n$-dimensional isotropic Bihom-ideal of $L$. By Lemma \ref{lemma3.1},  $[\beta(I),\beta(L),\alpha(I)]=0$. Let $B=L/I$ and $p: L \rightarrow B$ be the canonical projection. We can choose an isotropic complement subspace $B_{0}$ to $I$ in $L$, i.e. $L=B_{0}\dotplus I$ and $B_{0}\subseteq B_{0}^{\bot}$. Then $B_{0}^{\bot}=B_{0}$ since dim$B_0=n$.

Denote by $p_{0}$ (resp. $p_{1}$) the projection $L=B_{0}\dotplus I \rightarrow B_{0}$ (resp. $L=B_{0}\dotplus I\rightarrow I$) and let $q_{L}^{*}:I \rightarrow B^{*}$ is a linear map, where $q_{L}^{*}(i)(\bar{x}):= q_{L}(i,x),~\forall\, i\in I, \bar{x}\in B=L/I$.
 We claim that $q_{L}^{*}$ is a vector space isomorphism. In fact, if $\bar{x}=\bar{y}$, then $x-y\in I$, hence $q_{L}(i,x-y)\in q_{L}(I,I)=0$ and
 so $q_{L}(i,x)=q_{L}(i,y)$, which implies $q_{L}^{*}$ is well-defined and it is easy to see that $q_{L}^{*}$ is linear. If
 $q_{L}^{*}(i)=q_{L}^{*}(j)$, then $q_{L}^{*}(i)(\bar{x})=q_{L}^{*}(j)(\bar{x}), ~\forall\, x\in L$, i.e. $q_{L}(i,x)=q_{L}(j,x)$,
 which implies $i-j\in L^\bot=0$, hence $q_{L}^{*}$ is injective. Note that $\dim I=\dim B^*=n$, then $q_{L}^{*}$ is surjective.

In addition, $q_{L}^{*}$ has the following property, $\forall\, x,y,z\in L, i\in I$,
\begin{eqnarray*}
q_{L}^{*}([\beta(x),\beta(y),\alpha(i)])(\bar{\alpha}(\bar{z}))
&=&q_{L}([\beta(x),\beta(y),\alpha(i)],\alpha(z))\\
&=&-q_{L}(\alpha(i),[\beta(x),\beta(y),\alpha(z)])\\
&=&-q_{L}^{*}(\alpha(i))(\overline{[\beta(x),\beta(y),\alpha(z)]})\\
&=&-q_{L}^{*}(\alpha(i))([\overline{\beta(x)},\overline{\beta(y)},\overline{\alpha(z)}])\\
&=&-q_{L}^{*}(\alpha(i))\ad(\overline{\beta(x)},\overline{\beta(y)})(\overline{\alpha(z)})\\
&=&\ad^*(\overline{\beta(x)},\overline{\beta(y)})q_{L}^{*}(\alpha(i))(\overline{\alpha(z)}).
\end{eqnarray*}
A similar computation shows that
$$q_{L}^{*}([\beta(x),\beta(i),\alpha(y)])=-\ad^*(\overline{\beta(x)},\overline{\beta(y)})q_{L}^{*}(\alpha(i)),$$
$$q_{L}^{*}([\beta(i),\beta(x),\alpha(y)])=\ad^*(\overline{\beta(x)},\overline{\beta(y)})q_{L}^{*}(\alpha(i)).$$
Define a $3$-linear map
\begin{eqnarray*}
\theta:~~~B\times B\times B&\longrightarrow&B^{*}\\
(\bar{b_1},\bar{b_2},\bar{b_3})&\longmapsto&q_{L}^{*}(p_{1}([b_1,b_2,b_3])),
\end{eqnarray*}
where $b_1,b_2,b_3\in B_{0}.$ Then $\theta$ is well-defined since $p|_{B_0}$ is a vector space isomorphism.

Now define the bracket $[\cdot,\cdot,\cdot]_\theta$ on $B\oplus B^*$ by Proposition \ref{prop1}, we have $B\oplus B^*$
is a algebra. Let $\varphi:L \rightarrow B\oplus B^{*}$ be a linear map defined by $\varphi(x+i)=\bar{x}+q_{L}^{*}(i),~ \forall\, x+i\in B_0\dotplus I=L. $ Since $p|_{B_0}$ and $q_{L}^{*}$ are vector space isomorphisms, $\varphi$ is also a vector space isomorphism. Note that $\varphi\alpha(x+i)=\varphi(\alpha(x)+\alpha(i))=\overline{\alpha(x)}+q_{L}^{*}(\alpha(i))=\overline{\alpha(x)}+q_{L}^{*}(i)\bar{\alpha}
 =(\bar{\alpha}+\tilde{\bar{\alpha}})(\bar{x}+q_{L}^{*}(i))=(\bar{\alpha}+\tilde{\bar{\alpha}})\varphi(x+i)$, i.e. $\varphi\alpha=(\bar{\alpha}+\tilde{\bar{\alpha}})\varphi$. By the same way, $\varphi\beta=(\bar{\beta}+\tilde{\bar{\beta}})\varphi$. Furthermore, $\forall\, x,y,z\in L$, $i,j,k\in I$,
\begin{eqnarray*}
&&\varphi([\beta(x+i),\beta(y+j),\alpha(z+k)])\\
&=&\varphi([\beta(x)+\beta(i),\beta(y)+\beta(j),\alpha(z)+\alpha(k)])\\
&=&\varphi([\beta(x),\beta(y),\alpha(z)]+[\beta(x),\beta(y),\alpha(k)]+[\beta(x),\beta(j),\alpha(z)]+[\beta(x),\beta(j),\alpha(k)]\\
&&+[\beta(i),\beta(y),\alpha(z)]
+[\beta(i),\beta(y),\alpha(k)]+[\beta(i),\beta(j),\alpha(z)]+[\beta(i),\beta(j),\alpha(k)])\\
&=&\varphi([\beta(x),\beta(y),\alpha(z)]+[\beta(x),\beta(y),\alpha(k)]+[\beta(x),\beta(j),\alpha(z)]+[\beta(i),\beta(y),\alpha(z)])\\
&=&\varphi(p_{0}([\beta(x),\beta(y),\alpha(z)])\!+\!p_{1}([\beta(x),\beta(y),\alpha(z)])\!+\![\beta(x),\beta(y),\alpha(k)]\!+\![\beta(x),\beta(j),\alpha(z)]\\
&&+[\beta(i),\beta(y),\alpha(z)])\\
&=&\overline{[\beta(x),\beta(y),\alpha(z)]}+q^*_L\big(p_{1}([\beta(x),\beta(y),\alpha(z)])+[\beta(x),\beta(y),\alpha(k)]+[\beta(x),\beta(j),\alpha(z)]\\
&&+[\beta(i),\beta(y),\alpha(z)]\big)\\
&=&\overline{[\beta(x),\beta(y),\alpha(z)]}+\theta(\overline{\beta(x)},\overline{\beta(y)},\overline{\alpha(z)})+\ad^*(\overline{\beta(x)},\overline{\beta(y)})q_{L}^{*}(\alpha(k))\\
&&-\ad^*(\overline{\beta(x)},\overline{\beta(z)})q_{L}^{*}(\alpha(j))+\ad^*(\overline{\beta(y)},\overline{\beta(z)})q_{L}^{*}(\alpha(i))\\
&=&[\overline{\beta(x)}+q_{L}^{*}(\beta(i)),\overline{\beta(y)}+q_{L}^{*}(\beta(j)),\overline{\alpha(z)}+q_{L}^{*}(\alpha(k))]_\theta\\
&=&[\varphi(\beta(x)+\beta(i)),\varphi(\beta(y)+\beta(j)),\varphi(\alpha(z)+\alpha(k))]_\theta.
\end{eqnarray*}
Then $\varphi$ is an isomorphism of algebras, and $(B\oplus B^{*},[\cdot,\cdot,\cdot]_\theta, \bar{\alpha}+\tilde{\bar{\alpha}},\bar{\beta}+\tilde{\bar{\beta}})$ is a $3$-Bihom-Lie algebra.
Furthermore, we have
{\setlength\arraycolsep{2pt}
\begin{eqnarray*}
q_{B}(\varphi(x+i),\varphi(y+j))&=&q_{B}(\bar{x}+q_{L}^{*}(i),\bar{y}+q_{L}^{*}(j))\\
&=&q_{L}^{*}(i)(\bar{y})+q_{L}^{*}(j)(\bar{x})\\
&=&q_{L}(i,y)+q_{L}(j,x)\\
&=&q_{L}(x+i,y+j),
\end{eqnarray*}}
then $\varphi$ is isometric. And $\forall\, x,y,z,w\in L$, the relation
\begin{eqnarray*}
&&q_{B}([(\bar{\beta}+\tilde{\bar{\beta}})(\varphi(x)),(\bar{\beta}+\tilde{\bar{\beta}})(\varphi(y)),(\bar{\alpha}+\tilde{\bar{\alpha}})(\varphi(z))]_\theta,(\bar{\alpha}+\tilde{\bar{\alpha}})(\varphi(w)))\\
&=&q_{B}([\varphi(\beta(x)),\varphi(\beta(y)),\varphi(\alpha(z))]_\theta,\varphi(\alpha(w)))=q_B(\varphi([\beta(x),\beta(y),\alpha(z)]),\varphi(\alpha(w)))\\
&=&q_{L}([\beta(x),\beta(y),\alpha(z)],\alpha(w))=-q_{L}(\alpha(z),[\beta(x),\beta(y),\alpha(w)])\\
&=&-q_{B}(\varphi(\alpha(z)),[\varphi(\beta(x)),\varphi(\beta(y)),\varphi(\alpha(w))]_\theta)\\
&=&-q_{B}((\bar{\beta}+\tilde{\bar{\beta}})(\varphi(z)),[(\bar{\beta}+\tilde{\bar{\beta}})(\varphi(x)),(\bar{\beta}+\tilde{\bar{\beta}})(\varphi(y)),(\bar{\alpha}+\tilde{\bar{\alpha}})(\varphi(w))]_\theta),
\end{eqnarray*}
 which implies that $q_B$ is $\alpha\beta$-invariant. So $(B\oplus B^{*}, q_B,\bar{\beta}+\tilde{\bar{\beta}},\bar{\alpha}+\tilde{\bar{\alpha}})$ is a quadratic $3$-Bihom-Lie algebra.
Thus, the $T_\theta^*$-extension $(B\oplus B^{*}, q_B,\bar{\beta}+\tilde{\bar{\beta}},\bar{\alpha}+\tilde{\bar{\alpha}})$ of $B$ is isometric to $(L,q_{L},\alpha,\beta)$.
\end{proof}

\end{document}